\documentclass{amsart}
\usepackage[T1]{fontenc}
\usepackage{amsfonts}
\usepackage{amssymb}
\usepackage[colorlinks=true,linkcolor=blue,allcolors=blue
]{hyperref}
\usepackage{amsmath,blkarray,booktabs,bigstrut}
\usepackage{tikz}
\usepackage{tikz-cd}
\usepackage{graphicx}
\usepackage{dutchcal}
\usepackage{algorithm}
\usepackage{algpseudocode}
\usepackage{soul}

\numberwithin{equation}{section}

\newcommand\cE{{\mathcal E}}

\newcommand\cP{{\mathcal P}}

\newcommand\bA{{\bf A}}

\newcommand\bD{{\bf D}}
\newcommand\bL{{\bf L}}
\newcommand\bM{{\bf M}}
\newcommand\bO{{\bf O}}
\newcommand\bP{{\bf P}}

\newcommand\bU{{\bf U}}
\newcommand\bS{{\bf S}}
\newcommand\bSig{{\bf \Sigma}}
\newcommand\bT{{\bf T}}

\newcommand\R{{\mathbb R}}

\newcommand\N{{\mathbb N}}

\newcommand\im{{\operatorname{Im}}}

\newcommand\id{{\operatorname{\bf Id}}}
\newcommand\PFP{{\operatorname{PFP}}}

\newcommand\supp{{\operatorname{supp}}}
\renewcommand\min{{\operatorname{min}}}
\newcommand\pr{{\operatorname{pr}}}
\newcommand\tr{{\operatorname{tr}}}

\theoremstyle{plain}
  \newtheorem{theorem}[subsection]{Theorem}
  \newtheorem{proposition}[subsection]{Proposition}
  \newtheorem{lemma}[subsection]{Lemma}
  \newtheorem{corollary}[subsection]{Corollary}
  
  \newtheorem{problem}[subsection]{Problem}

\theoremstyle{definition}
  \newtheorem{definition}[subsection]{Definition}

\begin{document}
\include{psfig}
\title{Probabilistic frames and Wasserstein distances}

\author{Dongwei Chen and Martin Schmoll}
\address{Department of Mathematics, Colorado State University, CO, US, 80523}
\email{dongwei.chen@colostate.edu}
\address{School of Mathematical and Statistical Sciences, Clemson University, SC, US.}
\email{schmoll@clemson.edu}
\begin{abstract}
Probabilities with finite moments, in particular those which are probabilistic frames, are characterized and studied using p-Wasserstein distances.  
Adapting results from \cite{cuesta1996lower}, \cite{Gelbrich90} and \cite{olkin1982distance} to frame operators, 
we show that the sets of probabilistic frames with given frame operator, or equivalently given frame ellipsoid, 
are homeomorphic via an optimal linear push-forward. Applied to frame operators of couplings we study and generalize transport duals. 
In particular we obtain a non-degeneracy inequality identifying pairs of transport duals as frames. Finally we describe duals  
that do not arise as push-forwards and characterize those that are push-forwards.  
\end{abstract}

\subjclass[2020]{Primary 60A10, 42C15; Secondary 49Q22}

\maketitle
\section{Introduction and Main Results}
In recent work \cite{ehler2013probabilistic, loukili2020minimization, maslouhi2019probabilistic, wickman2023gradient} probabilistic frames, a set of Borel probability measures on $\R^n$ that generalize frames, have been considered. A frame in $\R^n$ is a finite set of vectors spanning $\R^n$, for a general background see \cite{christensen}. 
Probabilistic p-frames on $\R^n$ are Borel probability measures with finite p-th moments whose support, interpreted as set of vectors, spans $\R^n$  \cite{ehler2013probabilistic}. A key motivation to consider probabilistic frames is the possible use of analytic techniques and the advantage to compare frames of various cardinality by Wasserstein distances.  

Let us denote the set of Borel probability measures on $\R^n$ by $\cP(\R^n)$ and by $\cP_p(\R^n)$ those with finite p-th moments, i.e. $\cP_p(\R^n)=\{\mu \in \cP(\R^n):\ \int \|{\bf x}\|^p\ d \mu({\bf x})<\infty\}$.  
\begin{definition}
    $\mu \in  \mathcal{P}_p(\mathbb{R}^n)$ is called probabilistic p-frame if there exists  $0<A \leq B $ such that for any ${\bf x} \in \mathbb{R}^n$, 
\begin{equation*}
    A \Vert {\bf x}  \Vert^p \leq\int_{\mathbb{R}^n} | \left\langle {\bf x}, {\bf y} \right\rangle |^p\ d\mu({\bf y})  \leq B\Vert {\bf x}  \Vert^p. 
\end{equation*}
If, in addition $A = B$ we call $\mu$ tight, and if $A = B=1$ then $\mu$ is called Parseval (probabilistic) frame. 
\end{definition}
This standard definition of (probabilistic) frames does not directly provide a geometric intuition. Below we reformulate it as set of degeneracy distances using p-Wasserstein metrics. Briefly, the p-Wasserstein distance $W_p(\cdot, \cdot)$ provides a metric space structure on $\cP_p(\R^n)$ with convergence $\mu_n \rightarrow \mu$ being equivalent to weak-$\ast$ convergence together with convergence of the p-th moments: $\int \|{\bf x}\|^p\ d \mu_n \rightarrow \int \|{\bf x}\|^p\ d \mu$. Background on Wasserstein distances can be found in \cite{figalli-glaudo, villanitopics, villani2009optimal}.

To reformulate the frame condition for a Borel probability let $\pi_{{\bf x}^{\perp}}$ denote the orthogonal projection to the plane ${\bf x}^{\perp}$ perpendicular to ${\bf x}$ and $(\pi_{{\bf x}^{\perp}})_{\#}$ be the associated push-forward on measures. Letting $\cP_p({\bf x}^{\perp})$ denote the set of measures supported in ${\bf x}^{\perp}$ with finite $p$-th moment. 
\begin{proposition}\label{Was_Hyperplane_Frame}
For any unit-vector ${\bf x} \in S^{n-1}$ and any $p \geq 1$  
$$W_p(\mu, (\pi_{{\bf x}^{\perp}})_{\#}\mu)=\left(\int_{\mathbb{R}^n} |\langle{\bf x}, {\bf v} \rangle |^p d\mu({\bf v})\right)^{1/p}=\inf_{\nu \in \cP_p({\bf x}^{\perp}) } W_p(\mu, \nu).$$
\end{proposition}
Since a probabilistic p-frame spans $\R^n$, its support cannot lie in a proper linear subspace, so it must have positive p-Wasserstein distance to measures supported in any linear subspace: 
\begin{proposition}\label{Was_frame_formulation}
A measure $\mu \in \cP_p(\R^n)$ is a probabilistic p-frame if and only if $W_p(\mu,(\pi_{{\bf x}^{\perp}})_{\#}\mu)>0$ for all unit vectors ${\bf x} \in S^{n-1}$. It is a tight frame if and only if $W_p(\mu,(\pi_{{\bf x}^{\perp}})_{\#}\mu)=C>0$ for all ${\bf x} \in S^{n-1}$ and a Parseval frame 
if and only if $W_p(\mu,(\pi_{{\bf x}^{\perp}})_{\#}\mu)=1$ for all ${\bf x} \in S^{n-1}$.  
\end{proposition} 
Both propositions together (for proofs see section \ref{proof_prop_2_3}) indicate that the p-Wasserstein metrics capture the frame property and supply a geometric interpretation.  
Particularly interesting is the case $p=2$ where the respective Wasserstein distances are the eigenvalues of the {\em frame operator}. More precisely, if  
$\mu \in \cP_2(\R^n)$ and ${\bf x}, {\bf y} \in \mathbb{R}^n$  then $\langle {\bf x}, {\bf S}_{\mu} {\bf y}\rangle :=\int_{\mathbb{R}^n}  \langle{\bf x}, {\bf v}\rangle \langle{\bf y}, {\bf v}\rangle\  d\mu({\bf v})$ is a linear operator. Given a basis of $R^n$, the matrix ${\bf S}_{\mu}=\int_{\mathbb{R}^n} {\bf v} {\bf v}^t\  d\mu$ is a positive semi-definite, so that for ${\bf x} \in S^{n-1}$,   
\begin{equation} \label{int:Wasserstein_Frame}
W^2_2(\mu, (\pi_{{\bf x}^{\perp}})_{\#}\mu)={\bf x}^t \int_{\mathbb{R}^n} {\bf v} {\bf v}^t \ d\mu({\bf v})\ {\bf x}={\bf x}^t \ {\bf S_{\mu}}\ {\bf x}. \end{equation}
In particular $\bS_{\mu}$ is positive definite, if and only if $\mu$ is a (probabilistic) frame. 
Should $\mu$ not be a frame we still refer to $\bS_{\mu}$ as the frame operator. 
The {\em frame ellipsoid} $\cE_{\mu}:=\{{\bf S}^{1/2}_{\mu}\ {\bf x}: \  \|{\bf x}\|=1   \} \subset  \mathbb{R}^n$ 
associated with the root ${\bf S}^{1/2}_{\mu}$ of ${\bf S}_{\mu}$, is a hyperellipsoid exactly if ${\bf S}_{\mu}$ is definite, that is, if $\mu$ is a probabilistic frame. 
The frame ellipsoid provides the 2-Wasserstein distance of a given (probabilistic) frame to the closest non-frame in any given direction. 
It can be seen as a generalized version of the Legendre ellipsoid as defined in \cite{milman_pajor_82} 
for symmetric, convex and compact bodies in $\R^n$, even though we do not represent the ellipsoid as a body or mass distribution in $\R^n$.  

Let ${\mathbb S}^n_{+}$ be the set of non-negative definite $n\times n$ symmetric matrices and ${\mathbb S}^n_{++} \subset {\mathbb S}^n_{+}$ the 
subset of positive definite matrices. Further, let $\cP_{\bS} \subset \cP_2(\R^n)$ denote the set of probabilities having frame operator $\bS \in {\mathbb S}^n_{+}$ and define $ W_2(\mu, \cP_{\bS}):= \inf_{\nu \in \cP_{\bS}}  W_2(\mu, \nu) $. The lower estimate from \cite{cuesta1996lower} adapted to probabilistic frames shows that the characteristic Wasserstein distances given in Proposition \ref{Was_Hyperplane_Frame} are useful. Namely, if $\{{\bf v}_1,...,{\bf v}_n\}$ is an orthonormal basis of $\R^n$ and $\mu, \nu \in \cP_2(\R^n)$, then 
\begin{equation}
    W^2_2(\mu,\nu)\geq \sum^n_{i=1} \left(W_2(\mu, (\pi_{{{\bf v}_i}^{\perp}})_{\#}\mu) - W_2(\nu, (\pi_{{{\bf v}_i}^{\perp}})_{\#}\nu) \right)^{2}, 
\end{equation}
and equality holds if and only if $\nu=\bT_{\#}\mu$ where $\bT \in {\mathbb S}^n_+$ has eigenbasis $\{{\bf v}_i\}$.  
In the equality case, similar to \cite{Gelbrich90} and \cite{olkin1982distance} for covariance operators, we obtain:  \begin{theorem}\label{thm_main_1}
For any $\bS, \bA  \in {\mathbb S}^n_{++}$ the push-forward map  
$\bA_{\#}: \cP_{\bS} \rightarrow \cP_{\bA\bS\bA}$ is a homeomorphism,  
so that 
\begin{equation}\label{thm_main_1_Wasserstein_dist}
W^2_2(\mu, \bA_{\#}\mu)  = W^2_2(\mu, \cP_{\bA\bS\bA})=\tr\ {\bf S}({\bf \id}-\bA)^2 
\end{equation} 
for all $\mu \in \cP_{\bS}$ and $W_2(\mu, \nu)  > W_2(\mu, \bA_{\#}\mu)$ for any $\nu \in \cP_{\bA\bS\bA}$ 
so that $\nu \neq \bA_{\#}\mu$. 
\end{theorem}
As for the Wasserstein distance between frames with given frame operators, say $\bS, \bT  \in {\mathbb S}^n_{++}$, 
one applies Theorem \ref{thm_main_1} to the unique $\bA \in {\mathbb S}^n_{++}$ solving
$\bT=\bA\bS\bA$ (see Proposition \ref{prop_A_is_inverse}) given by 
\begin{equation}\label{AST}
    \bA=\bA(\bS,\bT):=\bS^{-1/2}(\bS^{1/2} \bT \bS^{1/2})^{1/2}\bS^{-1/2}.
\end{equation}
Since $ W_2(\mu, \cP_{\bT})$ depends only on $\bS$ and not on the particular $\mu \in \cP_{\bS}$, $d_W(\bS, \bT):=W_2(\cP_{\bS},\cP_{\bT})=W_2(\mu, \cP_{\bT})$ is well defined.  
\begin{proposition}\label{prop_metric_equivalence}
Given $\bS, \bT \in {\mathbb S}^n_{+}$. Then $d_W$ is a metric on  ${\mathbb S}^n_{+}$.  
More precisely, we have 
\begin{equation}\label{prop_eq_w_dist_symmetric}
d_W(\bS,\bT)=\tr (\bS +\bT - 2(\bS^{1/2}\bT\bS^{1/2})^{1/2}) 
\end{equation}
If $\|\cdot\|_{op}$ denotes the operator norm and $\|\cdot\|_{F}$ the Frobenius norm, then   
\begin{equation}\label{prop_eq_metric_norm_equivalence}  
\| \bS^{1/2}-\bT^{1/2} \|_{op} \leq W_2(\mathcal{P}_{\bf S}, \mathcal{P}_{\bf T})=d_W(\bS,\bT) \leq \| \bS^{1/2}-\bT^{1/2} \|_F.
\end{equation}
In particular, the topology generated by $d_W$ is the standard (norm) topology on ${\mathbb S}^n_{+}$. 
\end{proposition}
The proposition implies that the {\em frame map} $\mathcal{S}: \cP_2(\R^n) \rightarrow {\mathbb S}^n_{+}$ given
by $\mathcal{S}(\mu)=\bS_{\mu}$ is continuous; for a different proof, see \cite{wickman2023gradient}. 
The closely related metric $d(\bS, \bT):=\sqrt{d_W(\bS^2, \bT^2)}$ for symmetric matrices $\bS, \bT \in {\mathbb S}^n$ 
is by estimate \ref{prop_eq_metric_norm_equivalence} equivalent to norm-induced metrics, however, it
is not induced by a norm \cite{cuesta1996lower}.  
\begin{corollary} Let $p \in [1,\infty)$, then the set of probabilistic p-frames in $\cP_p(\R^n)$ is open in the p-Wasserstein topology on $\cP_p(\R^n)$.     
\end{corollary}
For $p=2$ this is easy to see, just compose the (continuous) frame map $\mathcal{S}$ with the determinant map $\det: {\mathbb S}^n_{+} \rightarrow \R_{\geq 0} $, so that $ \det \circ \mathcal{S}:  \cP_2(\mathbb{R}^n) \rightarrow \R_{\geq 0}$ is continuous. It follows, that the set of probabilistic frames 
$\cP_{++} := \{\mu \in \cP_2(\mathbb{R}^n): \det \circ \mathcal{S}(\mu)>0 \}= \mathcal{S}^{-1} {\mathbb S}^n_{++}$ is open. Hence, the frame map $\mathcal{S}: \cP_2(\R^n) \rightarrow {\mathbb S}^n_{+}$ defines a foliation on the set ${\mathbb S}^n_{+}$ of positive semidefinite symmetric $n \times n$ matrices. Restricted to frames, this gives a foliation $\mathcal{S}: \cP_{++} \rightarrow {\mathbb S}^n_{++}$ over the set of positive definite symmetric matrices. Theorem \ref{thm_main_1} implies that any two fibers $\cP_{\bS}, \cP_{\bT} \subset \cP_{++}$ are homeomorphic by optimal push-forward with $\bA(\bS, \bT) \in \mathbb{S}^n_{++}$. Further, we write $\cP_{+}:=\cP_{2}\backslash \cP_{++}$.\\

Given two probabilities $\mu,\nu \in \cP_2(\R^n)$ any coupling $\gamma \in \Gamma(\mu,\nu)$ lies in $\cP_2(\R^{2n})$, and hence  
has a frame operator $\bS_{\gamma} \in {\mathbb S}^{2n}_{+}$ displayed in \ref{cond:fix_offdiagoal_coup} below. 
Given $\bM \in {\rm GL}_n(\R)$, put
\begin{equation}
\label{cond:fix_offdiagoal_coup}  D(\bM):=\left\{(\mu,\nu) \in \cP^2_2(\R^n):\ \text{there is } \gamma \in \Gamma(\mu,\nu) \text{ with} \ 
{\bf S}_{\gamma}=\begin{bmatrix}
\bS_{\mu} & \bM \\
 \bM^t & \bS_{\nu}
\end{bmatrix}
\right \}
\end{equation} 
We call a pair $(\mu,\nu)\in D(\bM)$ an $\bM$-dual (pair). For  $\bM=\id$  the elements in $D(\id)$ are the well-known transport duals \cite{wickman2017duality}, 
where usually the set $D_{\mu}=D_{\mu}(\id)$ of transport duals to a fixed marginal $\mu$ is considered.  
We show that any of those sets are convex, see Proposition \ref{convexity_dual}, but unfortunately not closed nor compact, see Corollary \ref{non-compactness} and the example thereafter, so that a Choqu\'et representation of duals is not readily available. 
\begin{theorem}
Let $\bM \in {\rm GL}_n(\R)$ be definite (positive, or negative). If $(\mu,\nu) \in D(\bM)$ then we have  
$W_2(\mu, (\pi_{{\bf x}^{\perp}})_{\#}\mu)\cdot W_2(\nu, (\pi_{{\bf x}^{\perp}})_{\#}\nu) >0$   
for all ${\bf x} \in S^{n-1}$. In particular both $\mu$ and $\nu$ are frames. 

Furthermore the set of  $\bM$-duals  $D(\bM)$ is in bijection to the set of transport duals $D(\id)$, in particular it is not empty. For all transport duals $\nu$ of a given probabilistic frame $\mu$ we have $W_2(\nu, \mu) \geq  W_2(\cP_{\bS_{\nu}}, \cP_{\bS_{\mu}})$ 
and the inequality is an equality if and only if $\nu=(\bS^{-1})_{\#}\mu$ is the canonical dual of $\mu$. Moreover the canonical 
dual is the only transport dual with frame operator $\bS^{-1}$.   
\end{theorem}

\section{Applications and refinements of the main results}
Here is an application of Theorem \ref{thm_main_1}. 
For $\bT \in M_n(\R)$, let  $\R_+ \bT:=\{\lambda \bT: \lambda \in \R_+ \}$ be the ray through $\bT$. Take $\bT \in \mathbb{S}^n_{++}$ and consider $W^2_2(\mu,\R_+{\bf T} ):=\inf_{\lambda \in \R_+} W^2_2(\mu,\cP_{\lambda \bT} )$. 
\begin{corollary}
    Let $\mu$ be a probabilistic frame, then \\
    $W^2_2(\mu,\R_+{\bT} )=W^2_2(\mu, (c_{\min}  {\bf A}(\bS_{\mu}, \bT))_{\#}\mu)$   
    where  $c_{\min}=\frac{\tr \ (\bS^{1/2}_{\mu} \bT\bS^{1/2}_{\mu})^{1/2}}{\tr\ \bT} $. 
\end{corollary}
In particular the closest tight frame to a given frame is obtained by putting $\bT=\id $, see also \cite{loukili2020minimization}.     
\begin{proof}
From Theorem \ref{thm_main_1} we know that the probabilistic frame with frame operator $\lambda \bT$ closest to $\mu$ is given by $(\lambda^{1/2}  {\bf A}(\bS_{\mu}, \bT) )_{\#}\mu$. To determine the optimal $\lambda$, identity \ref{thm_main_1_Wasserstein_dist} implies    
\[W^2_2(\mu, (\lambda^{1/2}  {\bf A}(\bS_{\mu}, \bT))_{\#}\mu )=\tr\  \bS_{\mu} +\lambda \cdot \tr\ \bT -2 \sqrt{\lambda}\cdot \tr\ (\bS^{1/2}_{\mu}\bT \bS^{1/2}_{\mu})^{1/2}. \]
The right hand side is differentiable in $\lambda=c^2$, with minimum $c_{\min}$ as stated. 
\end{proof}

Proposition \ref{Was_Hyperplane_Frame} implies that if a unit vector ${\bf x}$ is an eigenvector of ${\bf S_{\mu}}$, 
then the corresponding eigenvalue is given by $W^2_2(\mu, (\pi_{{\bf x}^{\perp}})_{\#}\mu)$. 
Expanding a vector ${\bf x}=(x_1, ..., x_n)$ ($x_i=\langle {\bf x}, {\bf e}_i\rangle$) in an eigen-basis $\{{\bf e}_1,..., {\bf e}_n\}$ of ${\bf S_{\mu}}$ we obtain: 
\begin{corollary}\label{decomposition}  If ${\bf x}=(x_1, ..., x_n)$ is a unit vector in eigen-coordinates then 
$$  W^2_2(\mu, (\pi_{{\bf x}^{\perp}})_{\#}\mu)=\sum^n_{i=1} x^2_i \cdot W^2_2(\mu, (\pi_{{\bf e}_i^{\perp}})_{\#}\mu).
$$
In particular, if ${\bf S_{\mu}}$ is positive definite, then the vectors $\frac{{\bf x}}{W_2(\mu, (\pi_{{\bf x}^{\perp}})_{\#}\mu)}$, where ${\bf x}$ is a unit vector, lie on the ellipsoid
$$\sum^n_{i=1} x^2_i \cdot W^2_2(\mu, (\pi_{{\bf e}_i^{\perp}})_{\#}\mu)=1.
$$ 
\end{corollary} 
This Corollary may be applied to solve a basic case of the minimization problem for $p$-frame potentials. 
Recall that $W^p_p(\mu, (\pi_{{\bf x}^{\perp}})_{\#}\mu) \int_{(S^{n-1})}  \vert \left\langle {\bf x},{\bf y} \right\rangle \vert^p d\mu({\bf y})$. 
The probabilistic p-frame potential (PFP) is defined as 
\begin{equation*}
  \PFP(\mu,p)=  \int_{S^{n-1}} W^p_p(\mu, (\pi_{{\bf x}^{\perp}})_{\#}\mu) \ d\mu({\bf x}).  
\end{equation*}
The problem is then to determine $\PFP(p)=\underset{\mu \in \cP(S^{n-1})}{\inf}\PFP(\mu,p)$ and the probabilities $\mu$ that minimize the potential, i.e. those $\mu$ for which $\PFP(p)=\PFP(\mu,p)$.
Significant work has been done on this problem, see \cite{ehler2012minimization}, \cite{wickman2023gradient} and \cite{bilyk2022optimal}, 
but some questions are still open, for example in \cite{bilyk2022optimal} the authors conjecture that for $n \geq 2$ and $p>0$ not even, the optimizer is a finite discrete measure on $S^{n-1}$.  
Let us look at the case $p=2$, writing the 2-Wasserstein distance in terms of an orthogonal eigenbasis $\{{\bf v}_1,...,{\bf v}_n\}$ 
of the frame matrix, putting $x_i:=\langle {\bf x}, {\bf v}_i\rangle$ gives
\[ \int_{S^{n-1}} W^2_2(\mu, (\pi_{{\bf x}^{\perp}})_{\#}\mu) \ d\mu({\bf x})= \int_{S^{n-1}} \sum^n_{i=1}x_i^2 W^2_2(\mu, (\pi_{{\bf v}^{\perp}_i})_{\#}\mu) \ d\mu({\bf x})=\]
\[ \sum^n_{i=1} W^2_2(\mu, (\pi_{{\bf v}^{\perp}_i})_{\#}\mu) \int_{S^{n-1}} x_i^2 \ d\mu({\bf x})  = \sum^n_{i=1} W^4_2(\mu, (\pi_{{\bf v}^{\perp}_i})_{\#}\mu).
\]
We have used $ \int_{S^{n-1}} x_i^2 \ d\mu({\bf x}) = W^2_2(\mu, (\pi_{{\bf v}^{\perp}_i})_{\#}\mu)$. 
Since $\mu$ is a probability on the unit sphere integration of $\sum^n_{i=1}x^2_i=1$ gives the constraint $\sum^n_{i=1} W^2_2(\mu, (\pi_{{\bf v}^{\perp}_i})_{\#}\mu)=1$. The minimum is assumed when all distances $W_2(\mu, (\pi_{{\bf v}^{\perp}_i})_{\#}\mu)$ are equal. In particular the minimum for 2-frame potentials is assumed on a tight frame. \\
Ehler and Okoudjou (Proposition 4.7 and Corollary 4.8 in  \cite{ehler2012minimization}) observed essential properties of $\PFP$ minimizers. We present those here in terms of p-Wasserstein distances with a proof of Proposition 4.7 from \cite{ehler2012minimization} in terms of mass-transport and Wasserstein distances. 
\begin{proposition}\cite{ehler2012minimization}
If $\mu\in \cP(S^{n-1})$ minimizes the frame potential, i.e. $\PFP(\mu,p)= \PFP(p)$,  
then $ W^p_p(\mu, (\pi_{{\bf x}^{\perp}})_{\#}\mu)\geq \PFP(\mu)$ for all ${\bf x} \in S^{n-1}$ and equality holds for all ${\bf x} \in \supp(\mu)$. In particular $\mu$ is a frame. 
\end{proposition}
\begin{proof}
If we had $W_p(\mu, (\pi_{{\bf x}_-^{\perp}})_{\#}\mu)< W_p(\mu, (\pi_{{\bf x}_+^{\perp}})_{\#}\mu)$ for two distinct points ${\bf x}_- \in S^{n-1}$ and ${\bf x}_+ \in \supp\ \mu$, then we may take the mass located in a small neighborhood, say $K=K({\bf x}_+)$, of ${\bf x}_+$ and move it to the point ${\bf x}_-$. That way we obtain again a probability $\mu_-$ on $S^{n-1}$. One can easily see that $\PFP(\mu_-)< \PFP(\mu)$ contradicting the assumed minimality of $\PFP$ at $\mu$. 
Indeed, for any Borel set $E \subset S^{n-1}$ the signed measure $ \nu(E) :=\mu_-(E) -\mu(E) = \mu(K) \delta_{{\bf x}_-}-\mu(K \cap E)$  
obeys:  
$$ \int_{S^{n-1}} W^p_p(\mu, (\pi_{{\bf x}^{\perp}})_{\#}\mu)d\nu= \mu(K)\cdot W^p_p(\mu, (\pi_{{\bf x}^{\perp}_-})_{\#}\mu)-\int_{K} W^p_p(\mu, (\pi_{{\bf x}^{\perp}})_{\#}\mu)d\mu.
$$
Since $W_p(\mu, (\pi_{{\bf x}^{\perp}})_{\#}\mu)$ depends continuously on ${\bf x}$ one may shrink $K=K({\bf x}_+)$ if necessary, so that 
$W^p_p(\mu, (\pi_{{\bf x}^{\perp}})_{\#}\mu) \geq   \frac{1}{2}(W^p_p(\mu, (\pi_{{\bf x}^{\perp}_{+}})_{\#}\mu) + W^p_p(\mu, (\pi_{{\bf x}^{\perp}_{-}})_{\#}\mu)      ) $ holds for all ${\bf x} \in K$. Then 
$$ \int_{S^{n-1}} W^p_p(\mu, (\pi_{{\bf x}^{\perp}})_{\#}\mu)d\nu \leq \frac{1}{2}\mu(K)( W^p_p(\mu, (\pi_{{\bf x}^{\perp}_-})_{\#}\mu)- W^p_p(\mu, (\pi_{{\bf x}^{\perp}_{+}})_{\#}\mu) ) < 0.
$$
Now the integral on the left is the linear part of the expansion of $\PFP(\mu+\epsilon \nu)$ with respect to $\epsilon$. Since $\PFP(\mu)$ is a minimum 
this linear part cannot be negative and we have a contradiction. 

The first statement follows by taking ${\bf x}_- \notin \supp\ \mu$, the second if we consider 
${\bf x}_- \in \supp\ \mu$. It also follows that $\mu$ is a frame, since compactness of $\cP(S^{n-1})$ implies $\PFP(\mu,p)=\PFP(p)>0$. \end{proof}

Regarding $W_p(\mu, ( \pi_{{\bf x}^{\perp}})_{\#}\mu)$ as a function of ${\bf x}$, we record the identity: 
\begin{corollary}
For any unit-vector ${\bf x} \in S^{n-1}$ and any $p \geq 1$ we have 
\begin{equation}\label{delta-wasser}
W_p(\mu, ( \pi_{{\bf x}^{\perp}})_{\#}\mu)=W_p( (\pi_{{\bf x}})_{\#}\mu, \delta_{{\bf 0}})
\end{equation}
and 
$$W^2_2(\mu, \delta_{{\bf 0}})=W^2_2(\mu, (\pi_{{\bf x}})_{\#}\mu) + W_2^2( (\pi_{{\bf x}})_{\#}\mu, \delta_{{\bf 0}}).$$
\end{corollary}
\begin{proof}
Notice that for a unit-vector ${\bf x}$, one has $|\langle {\bf x}, {\bf v} \rangle|^p= \|\pi_{{\bf x}}{\bf v}\|^p =\|\pi_{{\bf x}}{\bf v}-{\bf 0}\|^p$.  
Together with $|\langle {\bf x}, {\bf v} \rangle|^p= \|\pi_{{\bf x}^{\perp}}{\bf v}-{\bf v}\|^p$ from Equation \ref{p-dist-identity} that gives: 
\begin{equation}
\begin{split}
W^p_p(( \pi_{{\bf x}})_{\#}\mu, \delta_{{\bf 0}})=\int_{\R} | v |^p   \ d (\pi_{{\bf x}})_{\#}\mu(v)=
\int_{\R^n} \|\pi_{{\bf x}}({\bf v})\|^p   \ d \mu({\bf v})=\\
=\int_{\R^n}\|{\bf v}- \pi_{{\bf x}^{\perp}}({\bf v})\|^p\ d \mu({\bf v})=W^p_p(\mu, (\pi_{{\bf x}^{\perp}})_{\#}\mu).
\end{split}
\end{equation}
The second statement is Pythagoras theorem. 
\end{proof}

\begin{proposition}
For $\bS \in {\mathbb S}^n_{+}$ the space $\cP_{\bS}$ is convex and hence pathwise connected. The same is true for $\cP_2$ and $\cP_{++}$. 
\end{proposition}
\begin{proof}
If $\mu, \nu \in \cP_{\bS}$ then by linearity of the integral the frame operator of $(1-t)\mu+t\nu$ is $\bS$ for every $t \in [0,1]$. 
This curve of measures is continuous in the weak-star topology. Since all involved measures have bounded second moments, it is also 
continuous in the $2$-Wasserstein topology. Convexity for $\cP_2$ and $\cP_{++}$ follows the same argument. 
\end{proof}
From the perspective of optimal transport (maps) this notion of connectedness has to be taken with care, since the given 
path between measures splits mass.  

Noticing that push-forward with $t \mapsto I_{\bA(\bS_{\mu}, \bT)}(t)$ defines a homotopy between $\cP_{++}$ and the fiber $\cP_{\bT}$ that 
is the identity on $\cP_{\bT}$ now gives: 
\begin{proposition}
For any $\bS \in {\mathbb S}^n_{++}$ $i:\cP_{\bS} \hookrightarrow \cP_{++}$ is a deformation retraction with 
respect to the retraction map $r: \cP_{++} \rightarrow \cP_{\bS}$ given by $r(\mu)=\bA(\bS_{\mu}, \bS)_{\#}\mu$. 
In particular the spaces $\cP_{++}$ and $\cP_{\bS}$ are homotopy equivalent.       
\end{proposition}
\begin{proof}
We note that all maps stated in the proposition are well-defined and continuous on Wasserstein space $(\cP_{++}, W_2)$. 
This is since push-forward with a continuous function is continuous.       
Now if $\mu \in \cP_{{\bS}}$ is $r(\mu)=\bA(\bS, \bS)_{\#}\mu=({\bf \id})_{\#}\mu=\mu$, so that  
$r \circ i = \id_{\cP_{{\bS}}}$.  All we need to show  is that $i \circ r = r$ is homotopic to the identity map on 
$\cP_{++}$. Such a homotopy is given by 
$$ H(t,\mu)= (I_{\bA(\bS_{\mu}, \bS)}(t))_{\#}\mu $$
for $(t, \mu) \in [0,1] \times \cP_{++}$. 
\end{proof}

Taking congruences of $\bS$ by $\bA\in {\mathbb S}_{++}$  defines an action of the multiplicative group 
$({\mathbb S}^n_{++},\cdot)$ on itself, denoted by $C_{\bS}(\bA):= \bA\bS\bA^t=\bA\bS\bA$. 
On the other hand push-forward with $\bA \in {\mathbb S}^n_{++}$ defines a group action on $\cP_{++}$ that commutes 
with the foliation map, so that: $C_{\bA} \circ S= S \circ \bA_{\#}$. In diagrams the stated action looks like: 
\begin{center}
\begin{tikzcd}{\mathbb S}^n_{++} \times \cP_{++} \arrow{rr} \arrow[d,"\id \times S" '] & &\cP_{++} \arrow{d}{S}\\
{\mathbb S}^n_{++} \times {\mathbb S}^n_{++}  \arrow{rr} & &{\mathbb S}^n_{++}
\end{tikzcd}
\hspace*{2cm}
\begin{tikzcd}(\bA, \mu) \arrow[rr, mapsto] \arrow[d, mapsto]{S} & &\bA_{\#}\mu \arrow[d, mapsto]{S}\\
(\bA, \bS_{\mu})  \arrow[rr, mapsto] & & \bA\bS_{\mu}\bA^t
\end{tikzcd}
\end{center}
To see this, we recall how frame operators transform under linear push-forwards, see \cite{loukili2020minimization}. 
Let $\bT$ be a linear transformation of $\R^n$, and $\mu$ be a probabilistic frame, then the frame operator of ${\bf T}_{\#}\mu$ is determined by     
\begin{equation}\label{fop-transform} 
\begin{split} & \langle {\bf x}, {\bf S}_{{\bf T}_{\#}\mu}{\bf x}\rangle =\int \langle{\bf x}, {\bf y}\rangle^2 \ d {\bf T}_{\#}\mu({\bf y})=
\int \langle{\bf x}, {\bf T} {\bf y}\rangle^2 \ d \mu({\bf y})=\\
&\int \langle {\bf T}^t{\bf x},  {\bf y}\rangle^2 \ d \mu({\bf y}) 
=\langle {\bf T}^t {\bf x}, {\bf S}_{\mu} {\bf T}^t {\bf x} \rangle=\langle {\bf x}, {\bf T} {\bf S}_{\mu} {\bf T}^t {\bf x} \rangle. 
\end{split}
\end{equation}
Since this identity holds for all ${\bf x}\in \R^n$, we have ${\bf S}_{{\bf T}_{\#}\mu}={\bf T} {\bf S}_{\mu} {\bf T}^t$. 
Because ${\bf S}_{\mu}$ is positive definite ${\bf S_{{\bf T}_{\#}\mu}}$ is always positive semi-definite. 
If ${\bf T}$ is invertible, then so is ${\bf S_{{\bf T}_{\#}\mu}}$. Summarizing the previous discussion gives: 
\begin{proposition}\label{prop_geod_act} 
The lift of the congruence action is faithful, continuous and minimizes distance with respect to $W_2$.  
In particular, push-forward with the interpolation maps $I_{\bA}(t):=(1-t){\bf  \id} + t\bA$ defines 2-Wasserstein constant speed geodesic curves $((I_{\bA}(t))_{\#}\mu)_{t \in [0,1]}$ in $(\cP_{++}, W_2)$.  
\end{proposition}
The proof is a simple consequence of Theorem \ref{thm_main_1}. The statement about interpolation geodesics is standard, see, for example, \cite{figalli-glaudo} Section 3.1.1.   

As a consequence one can work with the set of frames with a suitable frame operator. A typical choice would be to work with 
Parseval frames, since those are invariant under rotations.    
 

\section{Wasserstein openness of the set of probabilistic p-frames.}
We return to the setting of p-frames to show a general openness result.  
We start with a proof of Proposition \ref{Was_Hyperplane_Frame}. 
\begin{proof}[Proof of Proposition \ref{Was_Hyperplane_Frame}]\label{proof_prop_2_3}
First, note that by Cauchy-Schwarz the integral on the right is well defined for all $\mu \in \cP_p(\R^n)$. 
Now, for any unit vector ${\bf x} \in S^{n-1}$  
\begin{equation}\label{p-dist-identity}
|\langle {\bf x}, {\bf v} \rangle|^p=\text{dist}^p( {\bf v}, {\bf x}^{\perp})= \inf_{{\bf y} \in {\bf x}^{\perp}}\|{\bf y}-{\bf v}\|^p  
= \|\pi_{{\bf x}^{\perp}}({\bf v})-{\bf v}\|^p.
\end{equation} 
By definition
$$W^p_p(\mu, (\pi_{{\bf x}^{\perp}})_{\#}\mu)=\inf_{\gamma \in \Gamma(\mu, (\pi_{{\bf x}^{\perp}})_{\#}\mu)}
\int_{\R^{2n}} \|{\bf v}-{\bf y}\|^p\ d\gamma({\bf v},{\bf y}),
$$
but that minimum is taken on when pushing-forward the mass with the orthogonal projection onto ${\bf x}^{\perp}$, since that way every point moves minimal distance to target, hence 
$$W^p_p(\mu, (\pi_{{\bf x}^{\perp}})_{\#}\mu)=\int_{\R^{2n}}\|{\bf v}-{\bf y}\|^p\ d (\text{Id}\times \pi_{{\bf x}^{\perp}})_{\#}\mu=\int_{\R^n}\|{\bf v}- \pi_{{\bf x}^{\perp}}({\bf v})\|^p\ d \mu({\bf v}).
$$
Since $\text{supp}((\pi_{{\bf x}^{\perp}})_{\#}\mu) \subset {\bf x}^{\perp}$ we have 
 $ \inf_{\nu \in \cP_p({\bf x}^{\perp}) } W_p(\mu, \nu) \leq W_p(\mu, (\pi_{{\bf x}^{\perp}})_{\#}\mu)$. 
 
On the other hand, the  orthogonal projection $\pi_{{\bf x}^{\perp}}({\bf v})$ of any ${\bf v} \in \supp\ \mu$ minimizes the distance of ${\bf v}$  to ${\bf x}^{\perp}$,  therefore the push-forward of $\mu$ by  $\pi_{{\bf x}^{\perp}}$ minimizes the Wasserstein distance among all measures supported in ${\bf x}^{\perp}$. 
\end{proof}
\begin{proof}[Proof of Proposition \ref{Was_frame_formulation}]
This is a straight forward  consequence of propositon \ref{Was_Hyperplane_Frame}. For the characterization of tight and Parseval frames, observe that for any ${\bf x}\neq {\bf 0}  $:
$$  \int_{\R^n}|\langle {\bf v}, {\bf x} \rangle|^p\ d \mu({\bf v})= \|{\bf x} \|^p \int_{\R^n}|\langle {\bf v}, {\bf x}/ \|{\bf x} \| \rangle|^p\ d \mu({\bf v})=\|{\bf x} \|^p \ W^p_p(\mu, (\pi_{{\bf x}^{\perp}})_{\#}\mu).
$$
\end{proof}

\begin{proposition}[Openness of the set of probabilistic frames]
The set of probabilistic p-frames is open in the $W_p$ topology. 
\end{proposition}
\begin{proof} 
Suppose $\mu \in \cP_p(\R^n)$. Then by Proposition \ref{Was_Hyperplane_Frame},
$$W_p(\mu, (\pi_{{\bf x}^{\perp}})_{\#}\mu)=\inf_{\nu \in \cP_p({\bf x}^{\perp})}W_p(\mu, \nu)$$  
represents the p-Wasserstein distance of $\mu$ to ${\bf x}^{\perp}$.  
If $\mu$ is a probabilistic frame the p-Wasserstein distance between $\mu$ and any linear subspace, particularly the hyperplanes ${\bf x}^{\perp}$, must be positive.    
Indeed, if $\mu$ is a frame $\text{supp}\ \mu$ contains points in the complement of ${\bf x}^{\perp}$.  For any point in $\text{supp}\ \mu \cap ({\bf x}^{\perp})^C$ there exists a neighborhood disjoint to ${\bf x}^{\perp}$ that has positive $\mu$ mass, so that $\mu$ has positive p-Wasserstein distance to ${\bf x}^{\perp}$. 

Now, for a fixed probabilistic frame, say $\mu$, the p-Wasserstein distance to ${\bf x}^{\perp}$ depends continuously on the subspace ${\bf x}^{\perp}$, and therefore continuously on ${\bf x}$, in the topology induced by the p-Wasserstein metric. To see that take two vectors ${\bf x}, {\bf y} \in \R^n$, then the triangle inequality implies 
$$ | W_p(\mu, (\pi_{{\bf x}^{\perp}})_{\#}\mu)-  W_p(\mu, (\pi_{{\bf y}^{\perp}})_{\#} \mu)| \leq W_p((\pi_{{\bf x}^{\perp}})_{\#}\mu, (\pi_{{\bf y}^{\perp}})_{\#}\mu)
$$  
To estimate the Wasserstein distance on the right consider the coupling $\gamma:=D_{\#}\mu \in \Gamma(\mu, \mu)$, that is the push-forward of $\mu$ under the diagonal map $D({\bf x})=({\bf x},{\bf x}) \in \R^{2n}$.  Pushing this coupling forward with $\pi_{{\bf x}^{\perp}} \times \pi_{{\bf y}^{\perp}}$ gives a coupling 
of $(\pi_{{\bf x}^{\perp}})_{\#} \mu$ and $(\pi_{{\bf y}^{\perp}})_{\#} \mu$. Using this coupling gives the estimate 
\[W^p_p((\pi_{{\bf x}^{\perp}})_{\#}\mu, (\pi_{{\bf y}^{\perp}})_{\#}\mu) \leq \int_{\R^n} \| \pi_{{\bf x}^{\perp}}({\bf z}) -  \pi_{{\bf y}^{\perp}}({\bf z}) \|^p \ d\mu({\bf z})
\]
Using $\pi_{{\bf x}^{\perp}}({\bf z})={\bf z}-\langle{\bf z}, {\bf x} \rangle {\bf x}$ we obtain 
\[\int_{\R^n} \| \pi_{{\bf x}^{\perp}}({\bf z}) -  \pi_{{\bf y}^{\perp}}({\bf z}) \|^p \ d\mu=
\int_{\R^n} \| \langle{\bf z}, {\bf x} \rangle {\bf x}-\langle{\bf z}, {\bf y} \rangle {\bf y}    \|^p \ d\mu.
\]
Now put ${\bf y} = {\bf x} + \hat{ {\bf y}}$ to get $\langle{\bf z}, {\bf y} \rangle {\bf y} =\langle{\bf z}, {\bf x} + \hat{ {\bf y}} 
\rangle ({\bf x} + \hat{ {\bf y}}) = \langle{\bf z}, {\bf x} \rangle {\bf x} + \langle{\bf z}, \hat{ {\bf y}} 
\rangle {\bf x} + \langle{\bf z}, {\bf x} + \hat{ {\bf y}} \rangle \hat{{\bf y}}$. Hence 
\[
\int_{\R^n} \| \langle{\bf z}, {\bf x} \rangle {\bf x}-\langle{\bf z}, {\bf y} \rangle {\bf y}    \|^p \ d\mu
=\int_{\R^n} \|  \langle{\bf z}, \hat{ {\bf y}} \rangle {\bf x} + \langle{\bf z}, {\bf x} + \hat{ {\bf y}} \rangle \hat{{\bf y}}  \|^p \ d\mu\]
Minkowski's inequality followed by Cauchy-Schwarz while using $\|{\bf x}\|=1$ gives 
\[\leq 2^{p-1} \int_{\R^n} (\|  \langle{\bf z}, \hat{ {\bf y}} \rangle {\bf x} \|^p +\| \langle{\bf z}, {\bf x} + \hat{ {\bf y}} \rangle \hat{{\bf y}}  \|^p)  d\mu \leq 2^{p-1}  \|\hat{ {\bf y}}\|^p (1 +  \|{\bf x} + \hat{ {\bf y}}\|^p) 
\int_{\R^n} \|{\bf z}\|^p  \ d\mu.
\]
Since ${\bf y}= {\bf x}+ \hat{ {\bf y}}$ is a unit vector, we obtain 
\[
W^p_p((\pi_{{\bf x}^{\perp}})_{\#}\mu, (\pi_{{\bf y}^{\perp}})_{\#}\mu) \leq 2^p \|\hat{ {\bf y}}\|^p \int_{\R^n} \|{\bf z}\|^p  
 \ d\mu  =2^p M_p(\mu) \|{\bf y}-{\bf x}\|^p,
\]
which is continuity of the p-Wasserstein distance for projections. 
To conclude the argument we note that the space of $1$-codimensional subspaces in $\mathbb{R}^n$ is homeomorphic to    
$\mathbb{P}(\R^n)$ by identifying the subspace ${\bf x}^{\perp} $ with the projective line $ [{\bf x}]$ defined by any of its 
normal vectors $\pm {\bf x}$. Now $\mathbb{P}(\R^n)$ is compact and compactness implies that the continuous function ${\bf x} \mapsto W_p(\mu, (\pi_{{\bf x}^{\perp}})_{\#}\mu)$ takes on its minimum at some point, say ${\bf x}_{\min}$. 
We have already noticed that $W_p(\mu, (\pi_{{\bf x}^{\perp}})_{\#}\mu)>0$ for any subspace ${\bf x}^{\perp}$ of codimesion $1$, hence 
$$0 < c := W_p(\mu, (\pi_{{\bf x}^{\perp}_{\min}})_{\#}\mu) \leq W_p(\mu, (\pi_{{\bf x}^{\perp}})_{\#} \mu)$$ 
for all ${\bf x} \in S^{n-1}$. In particular the set $\{\nu \in \cP_p(\R^n):\  W_p(\nu, \mu) <  W_p(\mu, (\pi_{{\bf x}^{\perp}_{\min}})_{\#}\mu) \}$ 
is an open set of probabilistic p-frames. 
\end{proof}

\section{Wasserstein distances: Standard estimates and uniqueness}
The estimates displayed in this section are adapted versions of main results of \cite{Gelbrich90} and particularly \cite{cuesta1996lower} 
where instead of frame operators covariance operators are considered. The key arguments are almost the same. 
However the condition when the lower estimate for Wasserstein distances, stated before Theorem \ref{thm_main_1} in the introduction,   
is an equality is more direct and easier for probabilistic frames. This is because a frame operator is positive definite, 
while the covariance generally is not. Moreover, we do not need to consider centered measures. Recall from Equation \ref{AST}: $\bA(\bS,\bT)=\bS^{-1/2}(\bS^{1/2} \bT \bS^{1/2})^{1/2}\bS^{-1/2}$, so that 
$\bA^{-1}(\bS,\bT)=\bS^{1/2}(\bS^{1/2} \bT \bS^{1/2})^{-1/2}\bS^{1/2}$. 
These matrices have somewhat surprising properties that may not seem obvious at first glance. The next proposition and lemma will 
shed some light on some of those properties.  
\begin{proposition}\label{prop_A_is_inverse}
For any fixed $\bS \in {\mathbb S}^n_{++}$ the congruence map $C_{\bS}:{\mathbb S}^n_{+} \rightarrow {\mathbb S}^n_{+}$ given by $C_{\bS}(\bM) = \bM\bS\bM$, is bijective and its inverse is given by $C^{-1}_{\bS}(\bT)=\bA(\bS, \bT)$. In particular $\bA^{-1}(\bT,\bS)=\bA(\bS,\bT)$.
\end{proposition}
\begin{proof}
Note that the image of $C_{\bS}$ is always a positive semi-definite matrix.  For given $\bT \in {\mathbb S}^n_{+}$, let us solve $C_{\bS}(\bM)=\bT$, that is, solve $\bM\bS\bM=\bT$ for $\bM$. Since $\bS \in {\mathbb S}^n_{++}$ we can rewrite the previous equation as 
$$\bS^{1/2}\bT \bS^{1/2}=\bS^{1/2}\bM\bS\bM\bS^{1/2}=\bS^{1/2}\bM\bS^{1/2}\bS^{1/2}\bM\bS^{1/2}=(\bS^{1/2}\bM\bS^{1/2})^2.$$ 
Since $\bS^{1/2}\bT \bS^{1/2} \in {\mathbb S}^n_{+}$ we may take roots on either end. Solving for $\bM$ then gives 
$$\bM=\bS^{-1/2}(\bS^{1/2}\bT \bS^{1/2})^{1/2}\bS^{-1/2} = \bA(\bS,\bT) \in {\mathbb S}^n_{+}.$$ 
That the map is bijective follows from $C_{\bS}\circ C^{-1}_{\bS}(\bT)=C_{\bS}(\bA(\bS,\bT))=\bT$ and $C^{-1}_{\bS}\circ C_{\bS}(\bM)=A(\bS, \bM \bS \bM)=\bM$.  
The last identity follows from $\bS^{1/2} \bM \bS \bM \bS^{1/2} =(\bS^{1/2} \bM \bS^{1/2})^2$. 

For the last statement, with $\bA^{-1}(\bT,\bS)=\bT^{1/2}(\bT^{1/2} \bS \bT^{1/2})^{-1/2}\bT^{1/2}$ one easily verifies that $C_{\bS}(\bA^{-1}(\bT,\bS))=\bT$, because $C_{\bS}$ is a bijection the claim follows. 
\end{proof}

Given two probabilistic frames $\mu, \nu$ with frame operators  ${\bf S}_{\mu}$ and ${\bf S}_{\nu}$ 
let us write $ {\bf A}_{\mu, \nu} := {\bf A}({\bf S}_{\mu}, {\bf S}_{\nu})$. 
Recall, the {\em center of mass} or {\em mean} of a measure $\mu$ is the vector  
${\bf m}_{\mu}=\int_{\R^n} {\bf v}\  d\mu({\bf v})$. Then the {\em centered measure} of $\mu$ is given by 
$\overline{\mu}(A):=\mu(A+{\bf m}_{\mu})$ for any Borel set $A$.
Recall the covariance matrix of $\mu$ is given by ${\bf \Sigma}_{\mu}=\bS_{\overline{\mu}}$. Note, ${\bf \Sigma}_{\mu}$ is not necessarily invertible, i.e. $\bS_{\overline{\mu}}$ is not necessarily definite. 
In particular, a centered probabilistic frame may not be a probabilistic frame. In this case  ${\bf \bS}_{\mu}^{-1/2}$, respectively 
${\bf \Sigma}_{\mu}^{-1/2}$, is defined as a Moore-Penrose inverse. 
If ${\bf \Pi}_{\mu}$ is the (matrix version of the) orthogonal projection onto $\im\ \bS_{\mu}$, then the Moore-Penrose inverse has the property  $ {\bf \Pi}_{\mu}=\bS_{\mu}\bS^{-1}_{\mu}=\bS^{-1}_{\mu}\bS_{\mu}$. With that in mind we have $$ {\bf A}_{\overline{\mu},\overline{\nu}} = {\bf A}({\bf \Sigma}_{\mu}, {\bf \Sigma}_{\nu})= {\bf \Sigma}_{\mu}^{-1/2}({\bf \Sigma}_{\mu}^{1/2}{\bf \Sigma}_{\nu}{\bf \Sigma}_{\mu}^{1/2})^{1/2}{\bf \Sigma}_{\mu}^{-1/2}.$$ 
A special case of the first part of the following formula appeared in \cite{loukili2020minimization}. 
\begin{lemma}\label{Lemma:A_properties}
Let $\mu, \nu \in \cP_2(\R^n)$, not necessarily frames, then: 
\begin{enumerate}
    \item\   If $ \bS \in {\mathbb S}^n_{+}$ then ${\bf A}_{\mu, {\bf S}_{\#}\mu} = {\bf \Pi}_{\mu}{\bf S}{\bf \Pi}_{\mu}$, 
    and if $\mu$ is a frame then ${\bf A}_{\mu, {\bf S}_{\#}\mu} = \bS$. \label{lemma:push} 
    \vspace{1mm}
    \item\  If $\nu=({\bf A}_{\mu, \nu})_{\#}\mu$, then $({\bf \Pi}_{\overline{\mu}})_{\#}\overline{\nu}=({\bf A}_{\overline{\mu}, \overline{\nu}})_{\#}\overline{\mu}$. \label{lemma:cent_opt_cond}
\end{enumerate}
\end{lemma}
\begin{proof}
For the first statement, since $ {\bf S}^{1/2}_{\mu}{\bf S} {\bf S}_{\mu} {\bf S} {\bf S}^{1/2}_{\mu}= ({\bf S}^{1/2}_{\mu}{\bf S} {\bf S}_{\mu}^{1/2})^2$, $\bS^{1/2}_{\mu}$ symmetric and  $\text{Image } \bS_{\mu} = \text{Image } \bS^{1/2}_{\mu}$, we have 
$${\bf A}_{\mu, {\bf S}_{\#}\mu}={\bf S}^{-1/2}_{\mu}({\bf S}^{1/2}_{\mu}{\bf S} {\bf S}_{\mu} {\bf S} {\bf S}^{1/2}_{\mu})^{1/2}{\bf S}^{-1/2}_{\mu}= {\bf S}_{\mu}^{-1/2}{\bf S}^{1/2}_{\mu}{\bf S} {\bf S}_{\mu}^{1/2}{\bf S}_{\mu}^{-1/2} = {\bf \Pi}_{\mu}{\bf S}{\bf \Pi}_{\mu}.
$$
If $\mu$ is a frame, then $ \bS_{\mu} \in {\mathbb S}^n_{++}$, hence ${\bf \Pi}_{\mu}={\bf \id}$.

 
For the second identity, recall that $\overline{({\bf A}_{\mu, \nu})_{\#}\mu}=({\bf A}_{\mu, \nu})_{\#}\overline{\mu}$. Including $ ({\bf \Pi}_{\overline{\mu}})_{\#}\overline{\mu}=\overline{\mu} $ and the previous formula we get  
\[
\begin{split}
 ({\bf A}_{\overline{\mu}, \overline{ \nu}})_{\#}\overline{\mu}&=({\bf A}_{\overline{\mu}, (\overline{ {\bf A}_{\mu, \nu} )_{\#} \mu}})_{\#}\overline{\mu} = ({\bf \Pi}_{\overline{\mu}}{\bf A}_{\mu, \nu}{\bf \Pi}_{\overline{\mu}})_{\#}\overline{\mu}\\ 
 =&({\bf \Pi}_{\overline{\mu}})_{\#}({\bf A}_{\mu, \nu})_{\#}({\bf \Pi}_{\overline{\mu}})_{\#}\overline{\mu}= ({\bf \Pi}_{\overline{\mu}})_{\#}({\bf A}_{\mu, \nu})_{\#}\overline{\mu}=
 ({\bf \Pi}_{\overline{\mu}})_{\#} \overline{({\bf A}_{\mu, \nu})_{\#}\mu}.   
\end{split}
\] 
 
\end{proof}
Statement \ref{lemma:cent_opt_cond} of Lemma \ref{Lemma:A_properties} is to be expected as it verifies that the equality 
condition $\nu=({\bf A}_{\mu, \nu})_{\#}\mu$ for the respective Wasserstein distance estimates in the next Proposition \ref{prop:wasserstein_upper_est_and_eq_cond} and Proposition \ref{prop_universal_lin_push} below imply the equality condition for the respective estimate  
after centering the measures; $({\bf \Pi}_{\overline{\mu}})_{\#}\overline{\nu}=({\bf A}_{\overline{\mu}, \overline{\nu}})_{\#}\overline{\mu}$. 
See the respective estimates of \cite{Gelbrich90} and \cite{cuesta1996lower} using covariance operators.

\begin{proposition}\label{prop:wasserstein_upper_est_and_eq_cond}
For any unit vector ${\bf x}$ we have 
\begin{equation}
     W^2_2( (\pi_{{\bf x}})_{\#}\mu, (\pi_{{\bf x}})_{\#}\nu) \geq \left(W_2(\mu, (\pi_{ {\bf x}^{\perp}})_{\#}\mu) - W_2(\nu, (\pi_{{\bf x}^{\perp}})_{\#}\nu) \right)^{2}
\end{equation}
and if $\{{\bf e}_1,...,{\bf e}_n\}$ is an orthonormal basis then 
\begin{equation}\label{prop:est:was:gen_proj_rel}
    W^2_2(\mu,\nu)\geq \sum^n_{i=1} \left(W_2(\mu, (\pi_{{{\bf e}_i}^{\perp}})_{\#}\mu) - W_2(\nu, (\pi_{{{\bf e}_i}^{\perp}})_{\#}\nu) \right)^{2}, 
\end{equation}
equality holds if $\nu=\bT_{\#}\mu$ where $\bT \in {\mathbb S}^n_{+}$ diagonal with respect to $\{{\bf e}_i\}$.  
\end{proposition}
\begin{proof}
Abbreviating $\Gamma:=\Gamma(\mu, \nu)$ one has  
\begin{equation}\label{was_proj_est} 
\begin{split}
W^2_2(\mu,\nu)=\inf_{\gamma \in \Gamma}\int_{\R^n\times \R^n} \|{\bf x}-{\bf y}\|^2\ d\gamma = \inf_{\gamma \in \Gamma}\sum^n_{i=1}\int_{\R^n\times \R^n} |x_i-y_i|^2\ d\gamma\\ = \inf_{\gamma \in \Gamma}\sum^n_{i=1} \int_{\R \times \R} |x-y|^2\ d(\pi_{{\bf e}_i} \times \pi_{{\bf e}_i})_{\#}\gamma \geq \sum^n_{i=1} W^2_2( (\pi_{{\bf e}_i})_{\#}\mu, (\pi_{{\bf e}_i})_{\#}\nu).
\end{split}
\end{equation}
Now for any unit vector ${\bf z} \in S^{n-1}$ if $\gamma_{{\bf z}} \in \Gamma ( (\pi_{{\bf z}})_{\#}\mu, (\pi_{{\bf z}})_{\#}\nu)$ minimizes $W^2_2( (\pi_{{\bf z}})_{\#}\mu, (\pi_{{\bf z}})_{\#}\nu)$  then, by the reverse triangle inequality (in $L^2$): 
\begin{equation}\label{was_ltwo_est}
\begin{split} W^2_2( (\pi_{{\bf z}})_{\#}\mu, (\pi_{{\bf z}})_{\#}\nu)& =\int_{\R \times \R} |x-y|^2 \  d\gamma_{{\bf z}}\\
 \geq & \left(\left(\int_{\R} |x|^2\ d(\pi_{{\bf z}})_{\#}\mu\right)^{1/2} - \left(\int_{\R} |y|^2\ d (\pi_{{\bf z}})_{\#}\nu\right)^{1/2}\right)^{2}\\
=& \left(\left(\int_{\R^n} \langle {\bf x},{\bf z} \rangle^2\ d\mu\right)^{1/2} - \left(\int_{\R^n} \langle {\bf y},{\bf z} \rangle^2 \ d \nu\right)^{1/2}\right)^{2}\\
=& \left(W_2(\mu, (\pi_{{{\bf z}}^{\perp}})_{\#}\mu) - W_2(\nu, (\pi_{{{\bf z}}^{\perp}})_{\#}\nu) \right)^{2}, 
\end{split}
\end{equation}
This shows the first inequality stated. Using the estimate for ${\bf z}={\bf e}_i$ we obtain further 
\[W^2_2(\mu,\nu)\geq \sum^n_{i=1}\left(W_2(\mu, (\pi_{{{\bf e}_i}^{\perp}})_{\#}\mu) - W_2(\nu, (\pi_{{{\bf e}_i}^{\perp}})_{\#}\nu) \right)^{2}.   \]
 
Expanding square terms in Inequality \ref{was_ltwo_est} and using the marginals of $\gamma_i$ we obtain the equivalent condition   
\[
\int_{\R \times \R} xy \  d\gamma_i \leq \left(\int_{\R} x^2\ d(\pi_{{\bf e}_i})_{\#}\mu\right)^{1/2} \left(\int_{\R} y^2\ d (\pi_{{\bf e}_i})_{\#}\nu \right)^{1/2}. 
\]
This is a version of the Cauchy-Schwarz inequality with respect to $\gamma_i$. In particular, this inequality is an equality if $y=\lambda_i x$ for some $\lambda_i \geq 0$ and if the marginal measures agree. In this case $\gamma_i$ is a push-forward given by $\gamma_i=(1 \times \lambda_i )_{\#}(\pi_{{\bf e}_i})_{\#}\mu$. This map is optimal and hence equalizes also \ref{was_proj_est} for the respective coordinate, since the optimality condition is the same.  
More precisely, taking optimal scalings in each coordinate we see a linear map $\bT$ that is diagonal with respect to ${\bf e}_i$ and    
has $\lambda_i\geq 0$ as the $i$-th diagonal entry, implies equality in \ref{was_proj_est}. In other words, the optimal coupling $\gamma$ is a linear push-forward, that is optimal in every direction ${\bf e}_i$. Any such linear map is positive semi-definite. Directions with $\lambda_i=0$ may appear.  
\end{proof}
Proposition \ref{prop:wasserstein_upper_est_and_eq_cond} allows us to show the continuity of the frame map directly using Wasserstein distances (for a different argument, see \cite{wickman2023gradient}). 
\begin{corollary}\label{cor_continuity_fr_map}
 The frame map $\mathcal{S}: \cP_2(\mathbb{R}^n) \rightarrow {\mathbb S}^n_+$ is continuous in the Wasserstein topology and in the weak-$\ast$ topology, on $\mathcal{P}_2(\R^n)$. More precisely  
 $ \|{\bf S}^{1/2}_{\mu}-{\bf S}^{1/2}_{\nu}\|_{op}  \leq  W_2(\mu,\nu) $ with respect to the operator norm $\|\cdot\|_{op}$.    
 In particular $ \|{\bf S}^{1/2}-{\bf T}^{1/2} \|_{op}  \leq  W_2(\cP_{\bS},\cP_{\bT})=d_W(\bS,\bT) $.
\end{corollary}
\begin{proof}
 Take $\mu, \nu \in \cP_2(\mathbb{R}^n)$ with frame operators ${\bf S}_{\mu}$ and ${\bf S}_{\nu}$ respectively. 
 Let ${\bf x}$ be a unit vector, so that   
 $\sup_{\|\bf y\|=1} |{\bf y}^t({\bf S}^{1/2}_{\mu}-{\bf S}^{1/2}_{\nu}){\bf y}|=  | {\bf x}^t({\bf S}^{1/2}_{\mu}-{\bf S}^{1/2}_{\nu}){\bf x}|$.  
 Let $\{ {\bf e}_i\}$ be an orthonormal eigen-basis for ${\bf S}_{\mu}-{\bf S}_{\nu}$ and write ${\bf x} = \sum^n_{i=1} x_i {\bf e}_i$, then 
 \begin{equation*}
     \begin{split}
 \|{\bf S}^{1/2}_{\mu}-{\bf S}^{1/2}_{\nu}\|_{op}^2 = &\sup_{\|\bf y\|=1} |{\bf y}^t({\bf S}^{1/2}_{\mu}-{\bf S}^{1/2}_{\nu}){\bf y}|^2 \\ =& 
 \sum^n_{i=1}x^4_i({\bf e}^t_i({\bf S}^{1/2}_{\mu}-{\bf S}^{1/2}_{\nu}){\bf e}_i)^2  \leq 
 \sum^n_{i=1}(\langle{\bf e}_i,{\bf S}^{1/2}_{\mu}{\bf e}_i\rangle-\langle{\bf e}_i,{\bf S}^{1/2}_{\nu}{\bf e}_i\rangle)^2 \\
 =& \sum^n_{i=1}\left(W_2(\mu, (\pi_{{{\bf e}_i}^{\perp}})_{\#}\mu) - W_2(\nu, (\pi_{{{\bf e}_i}^{\perp}})_{\#}\nu) \right)^{2}
 \leq W^2_2(\mu,\nu).
 \end{split}
  \end{equation*}
 The last step is estimate \ref{prop:est:was:gen_proj_rel}. 
 We see $f(\mu):={\bf S}^{1/2}_{\mu}$ is continuous, hence $\mathcal{S}=f^2$ is continuous as well. The last statement is the definition of 
 $W_2(\cP_{\bS},\cP_{\bT})=d_W(\bS,\bT)$ in the introduction. 
 That shows the claim. 
\end{proof}

Recall that the p-th (central) moment $M_p(\mu)$ of a probability $\mu$ is given by $\int_{\R^n}\|x\|^p \  d \mu(x) $, if the integral is finite. 
Right from the definitions one easily confirms the well known formula 
\begin{equation}
    M_2(\mu)=\sum^n_{i=1} W^2_2(\mu, (\pi_{ {\bf e}_i^{\perp}})_{\#}\mu)= \tr\ {\bf S_{\mu}}
\end{equation}
for any orthonormal basis $\{\bf e_i\}$ of $\R^n$. 
Indeed 
$$
\tr\ {\bf S}_{\mu}=\sum^n_{i=1}\langle {\bf e}_i, {\bf S}_{\mu} {\bf e}_i\rangle =\sum^n_{i=1} \int_{\mathbb{R}^n}  \langle{\bf e}_i, {\bf v}\rangle^2  d\mu
=\int_{\mathbb{R}^n}  \|{\bf v}\|^2  d\mu({\bf v}) =M_2(\mu).
$$
The matrix version of the previous proposition gives Gelbrich's bound \cite{Gelbrich90} for frame operators. 
The proof is formally the same as Theorem 2.1 in \cite{cuesta1996lower}, we add it adapted to our conventions for convenience. 
\begin{corollary}[Gelbrich's bound \cite{Gelbrich90} for frame operators]\label{prop_lower_est}
Let $\mu,\nu \in \cP_{++}$ with respective frame operators ${\bf S}_{\mu}$ and ${\bf S}_{\nu}$, then  
\begin{equation}\label{matrix_ineq_lower_est}
    W^2_2(\mu,\nu)\geq \tr({\bf S}_{\mu}+{\bf S}_{\nu}-2({\bf S}_{\mu}^{1/2}{\bf S}_{\nu}{\bf S}_{\mu}^{1/2})^{1/2})= \tr\ {\bf S}_{\mu}({\bf \id}-\bA_{\mu, \nu})^2. 
\end{equation}
Equality holds if $\nu=(\bA_{\mu,\nu})_{\#}\mu$.
\end{corollary}
\begin{proof}
Given Inequality \ref{prop:est:was:gen_proj_rel} of Proposition \ref{prop:wasserstein_upper_est_and_eq_cond}, the statement will follow from the formula 
\[ \tr\ ({\bf S}_{\mu}^{1/2}{\bf S}_{\nu}{\bf S}_{\mu}^{1/2})^{1/2}=\sum^n_{i=1}\langle {\bf e}_i, {\bf S}_{\nu} {\bf e}_i\rangle^{1/2} \langle {\bf e}_i, {\bf S}_{\mu} {\bf e}_i\rangle^{1/2} 
\]
for some orthogonal basis $\{ {\bf e}_i \}$ of $\R^n$. Note, that the right hand side of Inequality \ref{matrix_ineq_lower_est} immediately follows from the right hand side of Inequality \ref{prop:est:was:gen_proj_rel} using 
$ W_2(\mu, (\pi_{{e_i}^{\perp}})_{\#}\mu) = \langle {\bf e}_i, {\bf S}_{\mu} {\bf e}_i\rangle^{1/2}$ and the respective formula for $\nu$. 
By Proposition \ref{prop_A_is_inverse} there is a unique $\bA_{\mu, \nu}=\bA(\bS_{\mu},\bS_{\nu})$ positive definite, so that  $\bS_{\nu}= \bA_{\mu, \nu} \bS_{\mu}\bA_{\mu, \nu}$. Let $\{ {\bf e}_i \}$ be an eigen-basis for $\bA_{\mu, \nu}$ with corresponding set of 
(positive) eigenvalues $\{\lambda_i\}$, then:   
 $$
 \langle {\bf e}_i, {\bf S}_{\nu} {\bf e}_i\rangle= \langle {\bf e}_i, (\bA_{\mu, \nu} \ {\bf S}_{\mu} \bA_{\mu, \nu}){\bf e}_i\rangle = \langle \bA_{\mu, \nu} {\bf e}_i, {\bf S}_{\mu} \bA_{\mu, \nu} {\bf e}_i\rangle = 
\lambda_i^2 \langle {\bf e}_i, {\bf S}_{\mu} {\bf e}_i\rangle.  
$$
Taking roots on both sides and using $\bA_{\mu, \nu}= {\bf S}_{\mu}^{-1/2}({\bf S}_{\mu}^{1/2}{\bf S}_{\nu}{\bf S}_{\mu}^{1/2})^{1/2} {\bf S}_{\mu}^{-1/2}$ from Proposition \ref{prop_A_is_inverse}, formal properties of the trace give the sought identity:  
\[ 
\begin{split}
     \tr\ ({\bf S}_{\mu}^{1/2}{\bf S}_{\nu}{\bf S}_{\mu}^{1/2})^{1/2} =& \ \tr\ ({\bf S}_{\mu}^{1/2} \bA_{\mu, \nu} {\bf S}_{\mu}^{1/2})
= \tr\ ({\bf S}_{\mu} \bA_{\mu, \nu} )=\\ 
=& \sum^n_{i=1}\lambda_i \langle {\bf e}_i, {\bf S}_{\mu} {\bf e}_i\rangle=\sum^n_{i=1}\langle {\bf e}_i, {\bf S}_{\nu} {\bf e}_i\rangle^{1/2} \langle {\bf e}_i, {\bf S}_{\mu} {\bf e}_i\rangle^{1/2}. 
\end{split}
\]
Putting this identity one obtains the stated estimate as follows
\[
\begin{split}
 W^2_2(\mu,\nu)\geq \tr &({\bf S}_{\mu}+{\bf S}_{\nu}-2({\bf S}_{\mu}^{1/2}{\bf S}_{\nu}{\bf S}_{\mu}^{1/2})^{1/2})= 
\sum^n_{i=1}(1-\lambda_i)^2 \langle {\bf e}_i, {\bf S}_{\mu} {\bf e}_i\rangle 
\\ 
&=\sum^n_{i=1} \langle {\bf e}_i,({\bf \id}-\bA_{\mu, \nu}) {\bf S}_{\mu}({\bf \id}-\bA_{\mu, \nu}) {\bf e}_i\rangle
= \tr\ {\bf S}_{\mu}({\bf \id}-\bA_{\mu, \nu})^2.
\end{split}
\]
By Proposition \ref{prop:wasserstein_upper_est_and_eq_cond} equality holds, if $\nu=(\bA_{\mu,\nu})_{\#}\mu$. 
\end{proof}
\subsection{Olkin and Pukelsheim's matrix problem}
For the use in the next section we add the approach of Olkin and Pukelsheim version \cite{olkin1982distance} of the optimality problem above.  
Historically that was the first contribution to the matrix minimization problem and it added a condition on the 
matrices that can appear as frame operators of couplings that we will use later. \\

Given two probabilistic frames $\mu$ with frame operator $\bS$ and $\nu$ with frame operator $\bT$ respectively, then 
\begin{equation*}
\begin{split}
 W^2_2(\mu, \nu)=&\inf_{\gamma} \int_{\R^{2n}} \|{\bf x}-{\bf y} \|^2 d \gamma({\bf x},{\bf y})\\ =&\int_{\R^n} \|{\bf x}\|^2 \ d\mu + \int_{\R^n} \|{\bf y}\|^2 \ d\nu-  2 \sup_{\gamma} \int_{\R^{2n}} \langle {\bf x}, {\bf y}\rangle\ d\gamma({\bf x},{\bf y})   
\end{split}    
\end{equation*} 
The frame operator of $\gamma \in \Gamma(\mu, \nu)$ is given by  
${\bf S}_{\gamma}=\int_{\R^{2n}} ({\bf x}, {\bf y})\cdot ({\bf x}, {\bf y})^t\ d \gamma({\bf x}, {\bf y})$, written   
in block matrix form it is   
\begin{equation}\label{frame_op_coupling}
\bS_{\gamma}=\begin{bmatrix}
\bS & {\bf \Psi} \\
{\bf \Psi}^t & \bT 
\end{bmatrix}, \ \text{ where }   {\bf \Psi}=\int_{\R^{2n}} {\bf x}\cdot {\bf y}^t \ d\gamma({\bf x},{\bf y}). 
\end{equation}
Note, that 
\[\tr \int_{\R^{2n}} {\bf x}\cdot {\bf y}^t \ d\gamma({\bf x},{\bf y}) = \int_{\R^{2n}} \langle {\bf x}, {\bf y}\rangle \ d\gamma({\bf x},{\bf y}), \] 
so that the previous equation for the Wasserstein distance implies for any coupling $\gamma \in \Gamma(\mu, \nu)$: 
\begin{equation}\label{Wasserstein_trace_estimate}
 W^2_2(\mu, \mathcal{P}_{\bf T}) \leq \tr ({\bf S} + {\bf T} - 2 {\bf \Psi}). 
\end{equation}
The matrix optimization problem is that given $\bT$ and $\bS$ positive semi-definite, determine ${\bf \Psi}$ in  
$\bS_{\gamma}$, as given by Equation \ref{frame_op_coupling}, so that $\tr \ {\bf \Psi}$ is maximal under the constraint that $\bS_{\gamma}$ be  
positive semi-definite. We will see below that an extreme $\bf \Psi$ arises via a frame matrix of a coupling and  
determines the Wasserstein distance by turning estimate \ref{Wasserstein_trace_estimate} into an equality. 
The statement and solution of this problem was presented by Olkin and Pukelsheim in \cite{olkin1982distance} based on a dualizing argument. 
We start by presenting the argument from Lemma 1 in \cite{olkin1982distance} providing a condition on the off-diagonal of the block matrix \ref{frame_op_coupling} for the block matrix to be positive semi-definite. Namely, if $\bS, \bT \in \mathbb{S}^n_{++}$, then  
\begin{equation}\label{block-matrix}
\begin{bmatrix}
\bS & {\bf \Psi} \\
 {\bf \Psi}^t & \bT 
\end{bmatrix} \in \mathbb{S}^n_{+},  
\end{equation}
is, using matrix congruence, equivalent to 
\[\begin{bmatrix}
\id & -{\bf \Psi}\bT^{-1} \\
 0 & \id 
\end{bmatrix}
\begin{bmatrix}
\bS & {\bf \Psi} \\
 {\bf \Psi}^t & \bT 
\end{bmatrix}
\begin{bmatrix}
\id & 0 \\
-\bT^{-1} {\bf \Psi}^t & \id 
\end{bmatrix}
=
\begin{bmatrix}
\bS- {\bf \Psi} \bT^{-1} {\bf \Psi}^t & 0 \\
 0 & \bT 
\end{bmatrix}
 \in \mathbb{S}^n_{+}, 
\]
and using a similar congruence equivalent to
\[\begin{bmatrix}
\bS & 0 \\
 0 & \bT - {\bf \Psi}^t \bS^{-1} {\bf \Psi}
\end{bmatrix}
 \in \mathbb{S}^n_{+}.
\]
Hence the initial matrix  is positive semi-definite if and only if either $\bS- {\bf \Psi} \bT^{-1} {\bf \Psi}^t  \in \mathbb{S}^n_{+}$ or  
$\bT- {\bf \Psi}^t \bS^{-1} {\bf \Psi} \in \mathbb{S}^n_{+}$. Because of symmetry, we will discuss only the first condition below, even though we utilize the second one in the section on transport duals. 
Since the trace is a linear function it is extreme on the boundary of the convex set $\{{\bf \Psi}: \bS- {\bf \Psi} \bT^{-1} {\bf \Psi}^t \in \mathbb{S}^n_{+}\}$. Convexity is easy to check using frame matrices. The boundary is the set of ${\bf \Psi}$ so that $\bS={\bf \Psi} \bT^{-1} {\bf \Psi}^t$. This is algebraically equivalent to $\bA^t\bS\bA=\bT$ with $\bA=\bS^{-1}{\bf \Psi}$. Note that there are many solutions $\bA$ to the equation $\bA^t\bS\bA=\bT$. However, any push forward of a probabilistic frame in $\cP_{\bS}$ with $\bA^t \in {\rm GL}_n(\R)$ that solves $\bA^t\bS\bA=\bT$ is a probabilistic frame in $\cP_{\bT}$. But we know among those push-forwards the one that maximizes $\tr\ {\bf \Psi}=\tr\ \bS \bA$ is $\bA=\bA(\bS,\bT)$ by Gelbrich's Theorem. Let us summarize this discussion. 
\begin{corollary}\label{olkin_pukelsheim_semd_condition}
Assume $\bS, \bT \in \mathbb{S}^n_{++}$, then the $2n \times 2n$ block matrix given by Equation \ref{block-matrix} is positive semi-definite, 
if and only if $\bS^{-1}- \bA \bT^{-1} \bA^t \in \mathbb{S}^n_{+}$ where $\bA = \bS^{-1}{\bf \Psi}$, or alternatively 
$\bT^{-1}- \bA^t \bS^{-1} \bA \in \mathbb{S}^n_{+}$ where $\bA = {\bf \Psi}\bT^{-1}$. 
A push-forward of $\mu \in \cP_{\bS}$ with any $\bA^t \in {\rm GL}_n(\R)$ induces a coupling with a marginal in $\cP_{\bT}$ where $\bT= \bA^t \bS \bA$, 
or equivalently $\bS^{-1}= \bA \bT^{-1} \bA^t$. 
\end{corollary}
\begin{proof}
Only the second and third statement need to be verified. That $(\bA^t)_{\#}\mu \in \cP_{\bT}$ with $\bT= \bA^t \bS \bA$ was shown earlier. 
By elementary algebra, this identity is equivalent to $\bS^{-1}= \bA \bT^{-1} \bA^t$. The same goes for the alternative identity.  
\end{proof}
{\bf Remark.} The identity $\bT= \bA^t \bS \bA$ implies that the frame operator of the coupling associated with push-forward by the linear map $\bA$ is not positive definite. Hence, such a coupling is never a probabilistic frame. This on the other hand is obvious since the graph of a linear map 
mapping $\R^n$ into $\R^n$ is a proper linear subspace. \\

Now we show that the optimal linear map in Gelbrich's estimate is the unique distance minimizing map between frames with prescribed frame operators.  
\begin{proposition}\label{prop_universal_lin_push}
Given ${\bf S} , {\bf T}$ in $\mathbb{S}^n_{++}$, then for every $\mu \in  \mathcal{P}_{\bf S}$ the push-forward $({\bf A}({\bf S}, {\bf T}))_{\#}\mu$ is the unique probabilistic frame in $\cP_{\bf T}$, so that
$W_2(\mu,  ({\bf A}({\bf S}, {\bf T}))_{\#}\mu )=W_2(\mu, \mathcal{P}_{\bf T})$.  
\end{proposition}

\begin{proof}
We extend an argument that was used in the special case of $\bT=\id$ in \cite{loukili2020minimization}. 
Consider the push forward $(\bA(\bS, \bT))_{\#}\mu$. Assume $\nu$ has frame operator $\bT$ and minimizes the 2-Wasserstein distance to $\mu$, so that $W_2(\mu, \nu)=W_2(\mu, \cP_{\bT})$. 
Let $\gamma$ be an optimal coupling between $\nu$ and $\mu$. Then its 
push forward by $\id \times \bA(\bS, \bT)$ is a coupling between $\nu$ and $(\bA(\bS, \bT))_{\#}\mu$ with frame operator 
\[ 
\begin{split}
{\bf S}_{({\bf \id} \times \bA(\bS, \bT))_{\#}\gamma}=&\begin{bmatrix}
{\bf \id} & 0 \\
0 & \bA(\bS, \bT)
\end{bmatrix}
\begin{bmatrix}
\bT & \bT \cdot \bA(\bT, \bS) \\
(\bT \cdot \bA(\bT, \bS))^t & \bS
\end{bmatrix}
\begin{bmatrix}
{\bf \id} & 0 \\
0 & \bA(\bS, \bT)
\end{bmatrix}
\\
=&
\begin{bmatrix}
\bT & \bT \cdot \bA(\bT, \bS) \cdot \bA(\bS, \bT) \\
(\bT\cdot \bA(\bT, \bS) \cdot \bA(\bS, \bT))^t & \bA(\bS, \bT) \cdot \bS \cdot \bA(\bS, \bT) 
\end{bmatrix}
=
\begin{bmatrix}
\bT & \bT  \\
\bT & \bT 
\end{bmatrix}
\end{split}
\]
so that 
\[ W^2_2( (\bA(\bS, \bT))_{\#}\mu, \nu) \leq \tr (\bT + \bT - 2 \bT)=0. 
\]
Hence $ (\bA(\bS, \bT))_{\#}\mu=\nu$.

\end{proof}


\subsection{Proofs of statements from the introduction}
\begin{proof}[Proof of Theorem \ref{thm_main_1}] 
Push-forward with a continuous map is continuous and in particular, if the push-forward is by a linear $\bA \in \mathbb{S}^n_{++}$ then push-forward with $\bA^{-1} \in \mathbb{S}^n_{++}$ provides a continuous inverse. 

The equation for the Wasserstein distance follows from Proposition \ref{prop_lower_est}. 
The particular shape of that formula for push-forwards with general positive definite matrices $\bA \in \mathbb{S}^n_{++}$ 
follows from the first identity in Lemma \ref{Lemma:A_properties}. Finally, the fact that push-forward with
$\bA \in \mathbb{S}^n_{++}$ is the only minimizer of the Wasserstein distance is shown in (the previous) Proposition \ref{prop_universal_lin_push}, again
using the first statement from Lemma \ref{Lemma:A_properties} to adapt to the situation stated in the theorem. 
\end{proof}
We are now in a position to show the following: 
\begin{proof}[Proof of Proposition \ref{prop_metric_equivalence}]
Recall that the ${\bf \Psi}$ so that $\tr\ {\bf \Psi}$ is maximal under the condition 
 \[ \label{proof_matrix} \begin{bmatrix}
{\bf S} & {\bf \Psi} \\
{\bf \Psi}^t & {\bf T}
\end{bmatrix} \in {\mathbb S}^{2n}_{+}\]
is given by ${\bf \Psi}=\bS^{1/2}(\bS^{1/2}\bT\bS^{1/2})^{1/2} \bS^{-1/2}$ 
so that $\tr (\Psi)=\tr (\bS^{1/2}\bT\bS^{1/2})^{1/2}$ is maximal, see \cite{Gelbrich90}, or alternatively \cite{olkin1982distance}. 
The matrix ${\bf \Psi}=\bS^{1/2}\bT^{1/2}$ obeys the identity ${\bf \Psi}\bT^{-1}{\bf \Psi}^t=\bS$, so that in particular $\bS-{\bf \Psi} \bT^{-1}{\bf \Psi}^t \geq 0$. By Corollary \ref{olkin_pukelsheim_semd_condition}, or Lemma 1 in \cite{olkin1982distance} this implies semi-definiteness, hence:   
\[ \begin{bmatrix}
{\bf S} & \bS^{1/2} \bT^{1/2}\\
(\bS^{1/2}\bT^{1/2})^t & {\bf T}
\end{bmatrix} \in {\mathbb S}^{2n}_{+}.\]
By maximality of $\tr (\bS^{1/2}\bT\bS^{1/2})^{1/2}$, see \cite{olkin1982distance}, we    
have $\tr\ \bS^{1/2}\bT^{1/2} \leq \tr (\bS^{1/2}\bT\bS^{1/2})^{1/2}$ 
and hence by Gelbrich's formula  
\[  
\begin{split} W^2_2(\mathcal{P}_{\bf S}, \mathcal{P}_{\bf T})& = \tr ({\bf S} + {\bf T} - 2 (\bS^{1/2}\bT\bS^{1/2})^{1/2} )\\ 
&\leq \tr ({\bf S} + {\bf T} - 2(\bS^{1/2}\bT^{1/2})) =\tr(\bS^{1/2}-\bT^{1/2})^2=\| \bS^{1/2}-\bT^{1/2} \|^2_F.  
\end{split}
\]
We add the arguments showing $d_W$ is a metric. Clearly $W_2(\cP_{\bT},\cP_{\bS})\geq 0$ and equality happens if and only if $\bT=\bS$. 
The symmetry is also clear, since $W_2$ is a metric. For the triangle inequality let $\bP \in {\mathbb S}^n_{++}$ and consider $\mu \in \cP_{\bP}$, 
then 
\[ 
\begin{split}
     W_2(\cP_{\bT},\cP_{\bS}) \leq  W_2(\bA(\bP&,\bT)_{\#} \mu , \bA(\bP,\bS)_{\#}\mu) \\ 
     & \leq  W_2(\bA(\bP,\bT)_{\#}\mu, \mu)+W_2(\mu, \bA(\bP,\bS)_{\#}\mu)\\ &  
\hspace*{3.3cm} = W_2(\cP_{\bT},\cP_{\bP})+W_2(\cP_{\bP},\cP_{\bS}). 
\end{split}
\]
Note that all norms on a finite-dimensional vector space are equivalent. The metrics $d_{op}(\bS, \bT):=\| \bS^{1/2}-\bT^{1/2} \|_{op}$, and $d_{F}(\bS, \bT):=\| \bS^{1/2}-\bT^{1/2} \|_{F}$ together with the lower estimate from Corollary \ref{cor_continuity_fr_map} complete the estimate.\\

For a symmetric representation of the metric, note that an optimal coupling between measures with frame operator $\bT$ and frame operator $\bS$ has frame operator with off-diagonal matrix $\Psi=\bT^{1/2}(\bT^{1/2}\bS\bT^{1/2})^{1/2} \bT^{-1/2}$.  
Clearly, from symmetry of $ W^2_2(\mathcal{P}_{\bf S}, \mathcal{P}_{\bf T})$ and Gelbrich's representation  
\[ \tr ({\bf T} + {\bf S} - 2 (\bT^{1/2}\bS \bT^{1/2})^{1/2} ) =  W^2_2(\mathcal{P}_{\bf S}, \mathcal{P}_{\bf T}) = \tr ({\bf S} + {\bf T} - 2 (\bS^{1/2}\bT\bS^{1/2})^{1/2} ).
\]
It follows $\tr\ (\bT^{1/2}\bS \bT^{1/2})^{1/2}  = \tr \ (\bS^{1/2}\bT\bS^{1/2})^{1/2}$, so that 
\[d_W(\bS,\bT)=\tr (\bS +\bT - (\bS^{1/2}\bT\bS^{1/2})^{1/2} - (\bT^{1/2}\bS \bT^{1/2})^{1/2} ). \]
Using $\tr\ (\bT^{1/2}\bS \bT^{1/2})^{1/2} = \tr\  (\bT \bA(\bT, \bS))$ and $\tr\ (\bS^{1/2}\bT\bS^{1/2})^{1/2}=
\tr\  (\bS \bA(\bS, \bT)) $ we may rewrite this as 
\[d_W(\bS,\bT)=\tr \ (\bS(\id - \bA(\bS, \bT)) + \bT(\id - \bA(\bT, \bS)) ). 
\]
The distance $d_W$ extends from $\mathbb{S}^n_{++}$ to $\mathbb{S}^n_{+}$ for continuity reasons. 
Indeed since the function 
$$(\bS,\bT) \mapsto \tr ({\bf S} + {\bf T} - 2 (\bS^{1/2}\bT\bS^{1/2})^{1/2} )$$
is well-defined and continuous on $\mathbb{S}^n_{+} \times \mathbb{S}^n_{+}$ and $\mathbb{S}^n_{+}$ is the closure of $\mathbb{S}^n_{++}$ 
the metric properties continue to hold on $\mathbb{S}^n_{+}$. 
\end{proof}

\section{Transport duals and generalizations.}
\subsection{Wasserstein distances, special couplings and transport duals} 
Let $\mu \in \cP_{2}(\R^n)$ be a probabilistic frame, following \cite{wickman2017duality} we define the set of {\em transport duals} for $\mu$ to be  
\begin{equation}
\label{cond:transport_dual}  D_{\mu}:=\left\{\nu \in \cP_2(\R^n):\ \text{there is } \gamma \in \Gamma(\mu,\nu) \text{ with } \int_{ \R^{2n} } {\bf x}{\bf y}^t\ d\gamma({\bf x},{\bf y}) = \id \right\}. 
\end{equation} 
An equivalent description of $D_{\mu}$ is:  
\begin{equation}
\label{cond:transport_dual_frame_op}  D_{\mu}=\left\{\nu \in \cP_2(\R^n):\ \text{there is } \gamma \in \Gamma(\mu,\nu) \text{ with } 
{\bf S}_{\gamma}=\begin{bmatrix}
\bS_{\mu} & {\bf \id} \\
{\bf \id} & \bS_{\nu}
\end{bmatrix}
\right \}.
\end{equation} 
The off-diagonal $n \times n$-matrices, here $\id$, of $\gamma$'s frame matrix are determined by the integral in Equation  \ref{cond:transport_dual}.  
This motivates the following generalization. 
\begin{definition}
Given $\bM \in GL_n(\R)$ and probabilities $\mu, \nu \in \cP_2(\R)$. We call $(\mu, \nu)$ a $\bM$-dual pair, 
if there is $\gamma \in \Gamma(\mu,\nu)$ with frame operator 
\begin{equation}\label{M_frame_matrix}
{\bf S}_{\gamma}=\begin{bmatrix}
\bS_{\mu} & \bM \\
\bM^t & \bS_{\nu}
\end{bmatrix}.\end{equation}
The set of $\bM$-dual pairs is denoted by $D(\bM)$. Further let $D_{\mu}(\bM) \subset D(\bM)$ be the set of $\bM$ dual couplings with fixed 
first marginal $\mu$. 
\end{definition}
\begin{theorem}
Suppose $\bM \in M_n(\R)$ is definite. 
Then for any $(\mu,\nu) \in D(\bM)$  we have for all ${\bf z} \in S^{n-1}$
\[W_2(\mu, (\pi_{{\bf z}^{\perp}})_{\#}\mu)\cdot W_2(\nu, (\pi_{{\bf z}^{\perp}})_{\#}\nu) >0.  \]
In particular, both $\mu$ and $\nu$ are frames. Moreover, the set of  $\bM$-duals  $D(\bM)$ is in bijection to the set of transport duals $D(\id)$, 
hence the set of $\bM$-duals is not empty. In fact 
push forward with $\bM^t$ as given by $(\mu, \nu)  \mapsto (\mu, (\bM^t)_{\#}\nu) \in D(\bM)$ defines a bijective map $D(\id) \rightarrow D(\bM)$.  
\end{theorem}
\begin{proof}
For any ${\bf z} \in S^{n-1}$ by definiteness of $\bM$    
\begin{equation}
\begin{split} 
0 < |\langle {\bf z}, \bM {\bf z}\rangle| =|\int_{\R^{2n}} \langle {\bf z},{\bf x} \rangle \langle {\bf y}, {\bf z} \rangle \ d\gamma(\mu, \nu)| \leq \int_{\R^{2n}} |\langle {\bf z},{\bf x} \rangle| \ |\langle {\bf z}, {\bf y} \rangle| \ d\gamma(\mu, \nu)\\
\leq (\int_{\R^{n}} |\langle {\bf z},{\bf x} \rangle|^2 \ d\mu)^{\frac{1}{2}}(\int_{\R^{n}} |\langle {\bf z}, {\bf y} \rangle|^2 \ d\nu)^{\frac{1}{2}}=W_2(\mu, (\pi_{{\bf z}^{\perp}})_{\#}\mu)\cdot W_2(\nu, (\pi_{{\bf z}^{\perp}})_{\#}\nu), 
\end{split}
\end{equation}
Would $\mu$ and $\nu$ be not both probabilistic frames, the right-hand side of this inequality would be zero for some ${\bf z} \in S^{n-1}$. 
Now let $(\mu, \nu)$ be a dual pair and $\gamma \in \Gamma(\mu, \nu)$ a coupling with the respective frame operator. Then the frame operator 
of the push forward of $\gamma$ with ${\bf \id} \times \bM^t$ is 
\[ 
{\bf S}_{({\bf \id} \times \bM^t)_{\#}\gamma}=\begin{bmatrix}
\id & 0 \\
0 & \bM^t
\end{bmatrix}
\begin{bmatrix}
\bS_{\mu} & \id \\
\id & \bS_{\nu}
\end{bmatrix}
\begin{bmatrix}
\id & 0 \\
0 & \bM
\end{bmatrix}
=
\begin{bmatrix}
\bS_{\mu} & \bM \\
\bM^t & \bM^t \cdot \bS_{\nu} \cdot \bM 
\end{bmatrix}
\]
hence 
\[ {\bf S}_{({\bf \id} \times \bM^t)_{\#}\gamma} = 
\begin{bmatrix}
\bS_{\mu} & \bM \\
\bM^t  & \bS_{(\bM^t)_{\#}\nu}  
\end{bmatrix}, 
\]
so that $(\mu, (\bM^t)_{\#}\nu) \in D(\bM)$.  
\end{proof}
An analogous calculation gives for any $n\times n$ matrix $\bA$
\[ 
{\bf S}_{(\bA \times \bA)_{\#}\gamma}
=
\begin{bmatrix}
\bA & 0 \\
0 & \bA
\end{bmatrix}
\begin{bmatrix}
\bS_{\mu} & \id \\
\id & \bS_{\nu}
\end{bmatrix}
\begin{bmatrix}
\bA^t & 0 \\
0 & \bA^t
\end{bmatrix}
=
\begin{bmatrix}
\bS_{\bA_{\#}\mu} & \bM \\
\bM  & \bS_{\bA_{\#}\nu}  
\end{bmatrix}
\]
with $\bM :=\bA \bA^t$ symmetric. In particular if $\bA \in {\rm O}(n)$, the set of orthogonal $n\times n$ matrices, then a pair of duals
is mapped to a pair of duals under this push-forward. If  
\[ 
{\bf S}_{(\bA^t \times \bA^{-1})_{\#}\gamma}
=
\begin{bmatrix}
\bA^t \cdot \bS_{\mu}\cdot \bA & \id \\
\id  & \bA^{-1} \cdot \bS_{\nu} \cdot (\bA^{-1})^t 
\end{bmatrix}=
\begin{bmatrix}
\bS_{(\bA^t)_{\#}\mu} & \id \\
\id  & \bS_{(\bA^{-1})_{\#}\nu} 
\end{bmatrix}
\]
with $\bA:=\bS^{-1/2}_{\mu}{\bf O} \bS^{1/2}_{\mu}$ 
and $\bO \in {\rm O}(n)$ the frame operator of the first marginal $\mu$ is stabilized. Applying the previous Theorem to transport duals gives:      
\begin{corollary}\label{frame_off_identity}
If  $(\mu,\nu) \in D(\id)$ then $\mu$ and $\nu$ are probabilistic frames and we have for all ${\bf z} \in S^{n-1}$
\[W_2(\mu, (\pi_{{\bf z}^{\perp}})_{\#}\mu)\cdot W_2(\nu, (\pi_{{\bf z}^{\perp}})_{\#}\nu) \geq 1.  \]
\end{corollary}
We write $\mu_c:=(\bS^{-1})_{\#}\mu$ for the canonical dual of a probabilistic frame $\mu$. 
\begin{theorem}
Given $\bS \in {\mathbb S}^n_{++}$ and $\mu \in \cP_{\bS}$. Then the canonical dual $\mu_c$ is the only
transport dual of $\mu$ with the frame operator $\bS^{-1}$ and for any non-canonical dual coupling $\gamma$ of $\mu$ and $\nu$ 
we have 
\[ \int_{\R^{2n}}\| {\bf x} -{\bf y} \|^2\ d\gamma({\bf x}, {\bf y} ) >  W^2_2(\mu,\mu_c).   \]
Furthermore for non-canonical dual pairs $W_2(\nu, \mu) >  W_2(\cP_{\bS_{\nu}}, \cP_{\bS})$, while equality holds for $\nu=\mu_c$. 

All eigenvalues of $\bS^{1/2}_{\nu}\bS_{\mu} \bS^{1/2}_{\nu}$, respectively 
$\bS_{\mu} \bS_{\nu}$, need to be at least $1$ for a transport dual between $\mu$ and $\nu$ to exist. If 
$\bS_{\nu} \neq \bS^{-1}_{\mu}$, then some of the eigenvalues of $\bS^{1/2}_{\nu}\bS_{\mu} \bS^{1/2}_{\nu}$, respectively $\bS_{\mu} \bS_{\nu}$, must be stricly greater than $1$. 
\end{theorem}
Following \cite{olkin1982distance}, given $\bA, {\bf B} \in \mathbb{S}_+$ we write $\bA \geq {\bf B}$ when $\bA - {\bf B} \in \mathbb{S}_+$ and $\bA > {\bf B}$ when $\bA - {\bf B} \in \mathbb{S}_{++}$ respectively. This is a partial order for positive semi-definite matrices and is known as Loewner order. 
\begin{proof}
Suppose $\nu \in  \cP_{\bS^{-1}}$ is a non-canonical transport dual of $\mu$, then for any dual coupling $\gamma \in \Gamma(\mu,\nu)$ 
its frame operator fulfills the inequality $\tr(\bS_{\mu}+\bS^{-1}_{\mu}-2\cdot \id) \geq W^2_2(\mu, \nu)$. Since $W^2_2(\mu, \nu)> W^2_2(\mu, \cP_{\bS^{-1}_{\mu}}) = W^2_2(\mu, (\bS^{-1}_{\mu})_{\#}\mu)=\tr(\bS_{\mu}+\bS^{-1}_{\mu}-2\cdot \id)$ we have a contradiction.   

By Corollary \ref{olkin_pukelsheim_semd_condition}, see also \cite[Lemma 1]{olkin1982distance}, a pair $(\mu, \nu) \in \cP_2(\mathbb{R}^n) \times \cP_2(\mathbb{R}^n)$ can only be a dual pair if $\bS_{\nu}-\bS^{-1}_{\mu} \in \mathbb{S}_+$. This implies $\tr \ \bS_{\nu}\geq \tr\ \bS^{-1}_{\mu}$ and since the non-negativity condition must hold 
for the frame operator of any dual coupling $\gamma \in \Gamma(\mu,\nu)$ the stated inequality follows from 
\[ \int_{\R^{2n}}\| {\bf x} -{\bf y} \|^2\ d\gamma({\bf x}, {\bf y} ) = \tr (\bS_{\mu}+\bS_{\nu}-2 \id) \geq \tr (\bS_{\mu}+\bS^{-1}_{\mu}-2 \id)= W^2_2(\mu,(\bS^{-1})_{\#}\mu).\]
By elementary algebra $\bS_{\nu}-\bS^{-1}_{\mu} \geq 0$ is equivalent to $\bS^{1/2}_{\nu}\bS_{\mu} \bS^{1/2}_{\nu} \geq \id$, or equivalently $(\bS^{1/2}_{\nu}\bS_{\mu} \bS^{1/2}_{\nu})^{1/2} \geq \id$, with equality if and only if $\bS_{\nu}=\bS^{-1}_{\mu}$. The second estimate now follows 
from $W^2_2(\cP_{\bS_{\mu}}, \cP_{\bS_{\nu}})=\tr(\bS_{\mu}+\bS_{\nu}- 2 (\bS^{1/2}_{\nu}\bS_{\mu} \bS^{1/2}_{\nu})^{1/2} )$. 
The previous definiteness condition implies that the eigenvalues of $\bS^{1/2}_{\nu}\bS_{\mu} \bS^{1/2}_{\nu}$ are at least $1$. 
If all eigenvalues of $\bS^{1/2}_{\nu}\bS_{\mu} \bS^{1/2}_{\nu}$ are one, then $\bS_{\nu} = \bS^{-1}_{\mu}$. Hence, if $\bS_{\nu} \neq \bS^{-1}_{\mu}$, then one of the eigenvalues of $\bS^{1/2}_{\nu}\bS_{\mu} \bS^{1/2}_{\nu}$ must be greater than $1$.  
Regarding the eigenvalues of $\bS_{\mu}\bS_{\nu}$, by multiplicativity of the determinant, $\lambda$ is an eigenvalue of $\bS^{1/2}_{\nu}\bS_{\mu} \bS^{1/2}_{\nu}$ is equivalent to $\det(\bS_{\mu} -\lambda \bS^{-1}_{\nu})=0$ and this is equivalent to $\det(\bS_{\mu}\bS_{\nu} -\lambda \id)=0$. 
\end{proof}

\begin{corollary} 
Assume $\mu \in \cP_{\bS}$. Then $W_2(\nu, (\pi_{{\bf x}^{\perp}})_{\#} \nu) \geq  W_2(\mu_c, (\pi_{{\bf x}^{\perp}})_{\#} \mu_c)$ for any transport dual 
$\nu \in \cP_{2}(\R^n)$ of $\mu$. Furthermore, if $W_2(\mu, (\pi_{{\bf x}^{\perp}})_{\#} \mu)=({\bf x}^t\bS_{\mu}{\bf x})^{1/2} \leq 1$
for all ${\bf x} \in S^{n-1}$, then $ W_2(\mu, \nu) \geq  W_2(\mu, \mu_c)$.

    
\end{corollary}
\begin{proof}
    For a (non canonical) transport dual of $\mu$ with frame-operator $\bS_{\nu}$ we necessarily have $\bS_{\nu}-\bS^{-1}_{\mu} \geq 0$, so that  
\[W^2_2(\nu, (\pi_{{\bf x}^{\perp}})_{\#} \nu)={\bf x }^t\bS_{\nu}{\bf x} \geq {\bf x }^t \bS^{-1}_{\mu} {\bf x} = W^2_2(\mu_c, (\pi_{{\bf x}^{\perp}})_{\#} \mu_c)\]
for all ${\bf x} \in S^{n-1}$. 
Furthermore we have 
\[  W_2(\nu, (\pi_{{\bf x}^{\perp}})_{\#} \nu) - W_2(\mu, (\pi_{{\bf x}^{\perp}})_{\#} \mu) \geq  W_2(\mu_c, (\pi_{{\bf x}^{\perp}})_{\#} \mu_c) -  W_2(\mu, (\pi_{{\bf x}^{\perp}})_{\#} \mu)  \]
for all ${\bf x} \in S^{n-1}$. 
If $W_2(\mu, (\pi_{{\bf x}^{\perp}})_{\#} \mu)=({\bf x}^t\bS_{\mu}{\bf x})^{1/2} \leq 1$
for all ${\bf x} \in S^{n-1}$, then $W_2(\mu_c, (\pi_{{\bf x}^{\perp}})_{\#} \mu_c)=({\bf x}^t\bS^{-1}_{\mu}{\bf x})^{1/2} \geq 1$
for all ${\bf x} \in S^{n-1}$, so that the right-hand side of the estimate is nonnegative. By estimate \ref{prop:est:was:gen_proj_rel} then   
$  W_2(\mu, \nu) \geq  W_2(\mu, \mu_c)$. Note that the right-hand of the above inequality equals $W_2(\mu_c, (\pi_{{\bf x}^{\perp}})_{\#} \mu_c)$ 
since $\mu_c$ is the push-forward of $\mu$ by $\bS^{-1} \in \bS^n_{++}$, so that inequality \ref{prop:est:was:gen_proj_rel} is an equality. 
\end{proof}



\subsection{Transport duals that arise by push-forward}\  
There is a standard construction of all duals to a given (finite) frame see \cite{christensen} Section 6.3. Page 159, which directly translates into the language
of probabilistic frames, see \cite{wickman2017duality}. More or less as in the finite case one shows that these are all transport duals obtained
by push-forward. Indeed, for a given probabilistic frame $\mu$ and ${\bf h} \in L^2(\R^n, \mu, \R^n):=\{{\bf f}=(f_1,...,f_n):\ \R^n \rightarrow \R^n:\ f_i \in L^2(\R^n, \mu)
\}$ define  
\begin{equation}\label{dual_push_forward} 
{\bf H}({\bf z}):= \bS^{-1}{\bf z} +{\bf h}({\bf z})-\int_{\R^n}\langle \bS^{-1}{\bf z},{\bf x} \rangle {\bf h}({\bf x})\ d\mu({\bf x}). 
\end{equation}
\begin{proposition}
All transport duals of $\mu \in \cP_2(\R^n)$ obtained by push-forward are given by ${\bf H}_{\#}\mu$ for some ${\bf h} \in L^2(\R^n, \mu, \R^n)$.   
\end{proposition}
\begin{proof}
The assumption ${\bf h} \in L^2(\R^n,\mu, \R^n)$ is necessary to have ${\bf h}_{\#}\mu\in \cP_2(\R^n)$ for given $\mu\in \cP_2(\R^n)$.
A standard verification shows that the push-forward ${\bf H}_{\#}\mu$ determines a transport dual and hence $\mu$ and ${\bf H}_{\#}\mu$ is a pair
of transport duals. Now, if push-forward of $\mu$ with  ${\bf h} \in L^2(\R^n,\mu, \R^n)$ defines a transport dual then  
\[\id = \int_{\R^{2n}} {\bf x}{\bf y}^t \ d ((\id \times {\bf h} )_{\#}\mu)({\bf x}, {\bf y}) =\int_{\R^n} {\bf x}{\bf h}({\bf x})^t \ d\mu({\bf x})=
(\int_{\R^n} {\bf h}({\bf x}) {\bf x}^t \ d\mu({\bf x}))^t, \] 
so that 
$$\bS^{-1}{\bf z} = \int_{\R^n} {\bf h}({\bf x}) {\bf x}^t \ d\mu({\bf x}) \cdot \bS^{-1}{\bf z}=\int_{\R^n}\langle \bS^{-1}{\bf z},{\bf x} \rangle {\bf h}({\bf x})\ d\mu({\bf x}).$$
The last identity implies ${\bf H} = {\bf h}$ and that shows the claim. 
  \end{proof}

\subsection{ Convexity properties of couplings and dual couplings} The following proposition is of
independent interest. It will allow us to construct transport duals using convex combinations.    
\begin{proposition}\label{convexity_dual}
The set of $\bM$-dual pairs with, or without, fixed first marginal is convex.  In particular $D_{\mu}(\bM)$ is a convex set.  
\end{proposition}
\begin{proof}
If couplings 
$\gamma_0 \in \Gamma(\mu_0,\nu_0)$ and $\gamma_1 \in \Gamma(\mu_1,\nu_1)$ are given, 
then $\gamma_t := (1-t)\gamma_0 +t\gamma_1 \in \Gamma(\mu_t,\nu_t)$
is a coupling between $\mu_t:=(1-t)\mu_0+t\mu_1$ and $\nu_t:=(1-t)\nu_0+t\nu_1$ for any $t \in [0,1]$. 
If both couplings are $\bM$-couplings then so is their convex combination: 
\[ \begin{split}\int_{ \R^{2n} } {\bf x}{\bf y}^t\ d\gamma_t({\bf x},{\bf y})=&(1-t)\int_{ \R^{2n} } {\bf x}{\bf y}^t\ d\gamma_0({\bf x},{\bf y})+t\int_{ \R^{2n} } {\bf x}{\bf y}^t\ d\gamma_1({\bf x},{\bf y})\\ =& (1-t)\bM+t\bM =\bM 
\end{split}.
\]
In terms of frame matrices   
\[   {\bf S}_{\gamma_t}=(1-t)\begin{bmatrix}
\bS_{\mu_0} & \bM \\
 \bM & \bS_{\nu_0}
\end{bmatrix} + 
t \begin{bmatrix}
\bS_{\mu_1} & \bM \\
 \bM & \bS_{\nu_1}
\end{bmatrix}
=\begin{bmatrix}
\bS_{\mu_t} & \bM \\
 \bM & \bS_{\nu_t}
\end{bmatrix}.
\]
Clearly the argument for fixed first marginal follows by putting $\mu_0=\mu_1$. 
That implies the convexity of $D_{\mu}$. 
\end{proof}

\subsection{Transport duals that do not arise by push-forward}  
Transport duals of probabilistic frames show phenomena different from duals of finite frames, essentially because couplings may split mass. 
For that reason, an inclusive description of transport duals cannot be achieved without considering mass splitting indirectly or directly. 
Here is an example.\medskip 

\noindent{\bf Example 1.} Suppose $\mu$ is a probabilistic frame on the line with frame operator $\bS_{\mu}=\lambda >0 $. Then, applying Equation \ref{dual_push_forward} in the one dimensional setting with $h(x)=\alpha$, i.e. pushing forward all mass to one point $\alpha \in \R$, we get the one parameter family of maps
\begin{equation}\label{dual_push_forw}
    {H}_{\alpha}(x)=\lambda^{-1}x + \alpha - \alpha \lambda^{-1} \int_{\R} xy\ d \mu(y)=\lambda^{-1}(1-\alpha m_{\mu})x + \alpha,  
\end{equation}
where $m_{\mu}$ denotes the center of mass of $\mu$. If push-forward with $\alpha$ is a transport dual and  
$m_{\mu} \neq 0$, then $\alpha=H_{m_{\mu}^{-1}}(x)=m_{\mu}^{-1}$. Hence $(H_{m_{\mu}^{-1}})_{\#} \mu = \delta_{m_{\mu}^{-1}}$. 
Since the coupling $\gamma:=(\text{id} \times H_{m_{\mu}^{-1}})_{\#}\mu \in \Gamma(\mu, \delta_{m_{\mu}^{-1}})$
is a dual coupling, by symmetry $\mu$ is a transport dual of $\delta_{m_{\mu}^{-1}}$ that is generally not a push-forward, since
mass is split. In fact, all probabilities $\mu$ with bounded second moments and center of mass $m_{\mu} \neq 0$ have $\delta_{m_{\mu}^{-1}}$ 
as transport dual.  
To summarize: 
\begin{proposition} \label{one_dim_transport_duals}
The set of transport duals of a delta mass $\delta_{a}$, $a\in \R \backslash \{0\}$, consists of all probabilistic frames $\mu \in \cP_2(\R)$ with
center of mass $m_{\mu}=a^{-1}$. 
A transport dual of $\delta_{a}$ is a push-forward, if and only if it is the canonical dual.  
\end{proposition}
\begin{proof}
Because of the previous statements, all that remains to be shown is that $\delta_{a}$  is a transport dual of $\mu$ when $m_{\mu}=a^{-1}$. 
By Equation \ref{dual_push_forward} we have $(H_{m^{-1}_{\mu}})_{\#}\mu=\delta_{a}$ and in that case $\delta_{a} \in D_{\mu}$.  
\end{proof}
\noindent Moreover, by using convex combinations of measures, it is easy to construct transport dual pairs where neither is the push-forward of
the other. Indeed, consider two non-canonical dual pairs, say $(\mu, \delta_{\alpha})$ and $(\nu, \delta_{\beta})$.   
Then the probabilities $\widetilde{\mu} = \frac{1}{2}(\mu + \delta_{\beta})$ and $\widetilde{\nu} = \frac{1}{2}(\nu + \delta_{\alpha})$ 
define a dual pair $(\widetilde{\mu},\widetilde{\nu})$ by convexity. That pair does not arise as push-forward in either direction. \\


\begin{corollary}
A point mass in $\R \backslash \{0\}$ and its canonical dual minimize the 2-Wasserstein distance
between the point mass and its transport duals. 
\end{corollary}
\begin{proof}
By Proposition \ref{one_dim_transport_duals}, the duals of $\delta_a$, $a \in \R\backslash \{0\}$, are the probabilities in $\cP_2(\R)$ with center of mass $a^{-1} \in \R\backslash \{0\}$. Let $\mu \in \cP_2(\R)$ be a transport dual to $\delta_{a}$. Since transport to a point is a push-forward we have $W^2_2(\delta_{a}, \mu)= \int (x-a)^2 \ d\mu$. 
Breaking this representation into terms we see, that the 2-Wasserstein distance between $\delta_a$ and $\mu$ depends only on the second moment $
\int x^2 \ d\mu$ of $\mu$. On the other hand  we have $\int (x - a^{-1})^2\ d \mu \geq 0$, hence $
\int x^2 \ d\mu \geq a^{-2} = (\int_{\R} x \ d \delta_{a^{-1}})^2$  so that $W^2_2(\delta_{a}, \mu) \geq \int (a^{-1}-a)^2 \ d\mu=W^2_2(\delta_{a}, \delta_{a^{-1}})$. 
\end{proof}
We can have as well duals to a point mass with any given  
\begin{corollary}
Given $a \in (0,1]$, then there exists a transport dual to $\delta_{a}$ with second moment (or frame operator) $\lambda$ for any $\lambda \geq a^{-2}$.      
\end{corollary}
\begin{proof}
By the previous proposition any probability with finite second moment centered at $a^{-1}$ is a transport dual of $\delta_a$. Using $\lambda \geq a^{-2}$ we 
see that  $\mu_{\lambda, a}:=\frac{a^2\lambda-1}{a^2\lambda}\delta_{0}+\frac{1}{a^2\lambda}\delta_{a\lambda}$ is a probability measure. 
The center of $\mu_{\lambda, a}$ is $a^{-1}$ and the second moment is $\frac{1}{a^2\lambda} \cdot a^2\lambda^2=\lambda$ as desired. 
\end{proof}

\noindent Convexity together with compactness of the set of transport duals would imply a classification of transport duals by Krein-Milman. 
Compactness unfortunately is not true:  
\begin{corollary} \label{non-compactness}
If $\mu \in \cP_2(\R)$ is a centered and compactly supported frame, then $D_{\mu}$ is not compact in the 2-Wasserstein topology.      
\end{corollary}
\begin{proof}
By assumption $m_{\mu}=0$, so that by Equation \ref{dual_push_forw} any push-forward of $\mu$ with $H_{\alpha}(x)=\lambda^{-1}x + \alpha$ is a transport dual. Recall $0< \lambda=\bS_{\mu}$ because $\mu$ is a frame.  Since the support of $\mu$ is compact, we can find a positive monotonically increasing sequence $\{\alpha_i\}_{i\in \N} \rightarrow \infty$,  
so that the measures $(H_{\alpha_i})_{\#}\mu$ have mutually disjoint support. By construction, the sequence $\{m_{(H_{\alpha_i})_{\#}\mu}\}$ of centers diverges, hence the sequence of transport duals $\{(H_{\alpha_i})_{\#}\mu\}$ diverges in $W_2$.  
\end{proof}

\noindent{\bf Example 2} (Non-canonical duals from convex combinations). Recall, a standard way to construct a non-canonical dual to an overcomplete finite frame, say $F \subset \R^n$ is to take the canonical dual $\widetilde{V}$ of a sub-frame, say $V \subset F$,  extending it by $0$ on the 
remaining vectors of $F$. 
One can extend this construction to the setting of probabilistic frames, by decomposing a probabilistic frame into 
a sub-frame for which one could take the canonical dual and a complementary mass that will be moved to the origin. 
Direct transfer of the finite frame construction of non-canonical duals would change the total mass of a probabilistic dual. 
Requiring the dual to be a probability requires a small change of the construction.   
Let us consider a one dimensional example. As earlier, we take a delta mass $\delta_a$ located at $a \in \R\backslash \{ 0 \}$. 
We split that mass into $\delta_a=\lambda \delta_a+ (1-\lambda)\delta_a$ for some $\lambda \in (0,1)$. 
Then if the dual of $\lambda \delta_a$  is a push-forward, it must be $\lambda \delta_{(\lambda a)^{-1}}$ and extended by the push-forward 
of $(1-\lambda)\delta_a$ to $(1-\lambda)\delta_0$ we obtain $m_{\lambda}=\lambda \delta_{(\lambda a)^{-1}} + (1-\lambda)\delta_0 $.  
In fact, the coupling $\gamma=\lambda(x,\lambda^{-1}x^{-1})_{\#}\delta_a + (1-\lambda)(x,0)_{\#}\delta_a$ with marginals $m_{\lambda}$ and $\delta_a$ 
shows that the marginals are a dual pair for any $\lambda \in (0,1)$. In particular, the family of duals $\{m_{\lambda}:\ \lambda \in (0,1)\}$  
has  $\delta_0$ as a weak-star limit point, so that the set of transport duals is generally not weakly closed. Note, that 
the second moments of $m_{\lambda}$ diverge as $\lambda \rightarrow 0$, hence the convergence to $\delta_0$ does not hold 
with respect to the Wasserstein metric.

\subsection{Transport duals by convex combinations}

\noindent Below we describe dual couplings using generalized convex combinations, those are probability measures on the space of couplings. 
Recall that the frame map $\bS: \gamma \mapsto \bS_{\gamma}$ is continuous in $\gamma$ and so are the maps decomposing
the frame operator further into four $n\times n$ matrices $\bU: \gamma \mapsto \bU_{\gamma}$, $\bL: \gamma \mapsto \bL_{\gamma}$ (upper and lower diagonal) and $\bM: \gamma \mapsto \bM_{\gamma}$ (the off-diagonal), the fourth map being $\bM^t$ is determined by $\bM$ and clearly continuous.   
Below $M_n(\R)$ denotes the set of real valued $n\times n$ matrices. 
More precisely, let $\cP_c=\cP_c(\R^{2n})$ denote the set of couplings on $\R^{2n}$ with marginals in $\cP_2(\R^n)$ and  
$\xi$ be a probability on $\cP_c$, so that
\begin{enumerate}
    \item $\int_{M_n(\R)}\bA\ d (\bU)_{\#}\xi(\bA)=\int_{\cP_c}\bU_{\gamma}\ d \xi(\gamma) = \bS_{\mu}$
    \item $\int_{M_n(\R)}\bA\ d (\bL)_{\#}\xi(\bA)=\int_{\cP_c}\bL_{\gamma}\ d \xi(\gamma) = \bS_{\nu}$ 
    \item $\int_{\cP_c}(\int_{\R^n} {\bf x} {\bf y}^t  \ d \gamma({\bf x}, {\bf y}) )\  d \xi(\gamma)=\bM$.
\end{enumerate}
Here $\bM \in M_n(\R)$ is any given off-diagonal matrix, so that the frame operator $\bS_{\xi}=\int_{\cP_c}\bS_{\gamma}\ d \xi(\gamma)$ 
of $\xi$ is positive definite. Then $\xi$ defines the $\bM$-dual pair 
\[ (\mu_{\xi}, \nu_{\xi})=((\pr_{{\bf x}})_{\#} \int_{\cP_c} d  \xi(\gamma),(\pr_{{\bf y}})_{\#} \int_{\cP_c} d  \xi(\gamma) ).
\]
For transport duals put $\bM=\id$.   
More background on the space of measures can be found in \cite{einsiedler}. 
Clearly any transport dual defines a probability on the space of couplings, just take the delta measure on the particular coupling. 
On the other hand property (1) and (2) imply that $\xi$ is a coupling between $\mu$ and $\nu$, that is a transport dual when (3) holds  
for $\bM=\id$. One question is, if all ${\bf M}$-duals can be represented by measures on the space of couplings and if this is a practical 
representation for applications.


\section{Concluding Questions and Remarks}
The transport duals and hence the ${\bf M}$-duals are not yet fully understood. It seems that there are always 
non-frames in the closure of the set of dual frames, so the question arises if one can convex compose every transport dual
from lower-dimensional distributions. We did not restrict our exposition to absolutely continuous measures at
any point, since this would be an obstruction to making conclusions about discrete measures, i.e. standard frames. In this direction it 
would be beneficial, to show fiber-connectedness for finite frames of fixed cardinality. 
Furthermore, it would be important to study the case of infinite-dimensional Hilbert spaces.





\bibliographystyle{plain} 
\bibliography{refs}

\begin{thebibliography}{10}

\bibitem{bilyk2022optimal}
Dmitriy Bilyk, Alexey Glazyrin, Ryan Matzke, Josiah Park, and Oleksandr
  Vlasiuk.
\newblock Optimal measures for p-frame energies on spheres.
\newblock {\em Rev. Mat. Iberoam.}, 2022.

\bibitem{christensen}
Ole Christensen.
\newblock {\em An introduction to frames and Riesz bases}, volume~87 of {\em
  Applied and Numerical Harmonic Analysis}.
\newblock Birkh{\"a}user, 2016.

\bibitem{cuesta1996lower}
Juan~Antonio Cuesta-Albertos, C~Matr{\'a}n-Bea, and A~Tuero-Diaz.
\newblock On lower bounds for the $l^2$-wasserstein metric in a hilbert space.
\newblock {\em Journal of Theoretical Probability}, 9(2):263--284, 1996.

\bibitem{ehler2012minimization}
Martin Ehler and Kasso~A Okoudjou.
\newblock Minimization of the probabilistic p-frame potential.
\newblock {\em Journal of Statistical Planning and Inference}, 142(3):645--659,
  2012.

\bibitem{ehler2013probabilistic}
Martin Ehler and Kasso~A Okoudjou.
\newblock Probabilistic frames: an overview.
\newblock {\em Finite frames}, pages 415--436, 2013.

\bibitem{einsiedler}
Manfred Einsiedler and Thomas Ward.
\newblock {\em Ergodic Theory}, volume 259 of {\em Graduate Texts in Math.}
\newblock Springer, 2011.

\bibitem{figalli-glaudo}
Alessio Figalli and Federico Glaudo.
\newblock {\em An Invitation to Optimal Transport, Wasserstein Distances, and
  Gradient Flows}.
\newblock EMS Press, 2021.

\bibitem{Gelbrich90}
Matthias Gelbrich.
\newblock On a formula for the $l^2$ wasserstein metric between measures on
  euclidean and hilbert spaces.
\newblock {\em Mathematische Nachrichten}, 147(1):185--203, 1990.

\bibitem{maslouhi2019probabilistic}
S~Loukili and M~Maslouhi.
\newblock Probabilistic tight frames and representation of positive
  operator-valued measures.
\newblock {\em Applied and Computational Harmonic Analysis}, 47(1):212--225,
  2019.

\bibitem{loukili2020minimization}
S~Loukili and M~Maslouhi.
\newblock A minimization problem for probabilistic frames.
\newblock {\em Applied and Computational Harmonic Analysis}, 49(2):558--572,
  2020.

\bibitem{olkin1982distance}
Ingram Olkin and Friedrich Pukelsheim.
\newblock The distance between two random vectors with given dispersion
  matrices.
\newblock {\em Linear Algebra and its Applications}, 48:257--263, 1982.

\bibitem{milman_pajor_82}
Milman V.D. and A~Pajor.
\newblock Isotropic position and inertia ellipsoids and zonoids of the unit
  ball of a normed n-dimensional space.
\newblock {\em Lecture Notes in Mathematics}, 1376:64--104, 1989.

\bibitem{villanitopics}
C{\'e}dric Villani.
\newblock {\em Topics in optimal transportion}, volume~58 of {\em Graduate
  Studies in Math.}
\newblock American Mathematical society, 2003.

\bibitem{villani2009optimal}
C{\'e}dric Villani.
\newblock {\em Optimal transport: Old and New}, volume 338 of {\em Grundlehren
  der mathematischen Wissenschaften}.
\newblock Springer, 2009.

\bibitem{wickman2017duality}
Clare Wickman and Kasso Okoudjou.
\newblock Duality and geodesics for probabilistic frames.
\newblock {\em Linear Algebra and Its Applications}, 532:198--221, 2017.

\bibitem{wickman2023gradient}
Clare Wickman and Kasso~A Okoudjou.
\newblock Gradient flows for probabilistic frame potentials in the wasserstein
  space.
\newblock {\em SIAM Journal on Mathematical Analysis}, 55(3):2324--2346, 2023.

\end{thebibliography}

\end{document}